\documentclass[11pt]{article}
\usepackage{amsfonts, amsthm, amssymb, amsmath, stmaryrd}
\usepackage{mathrsfs,array}
\usepackage{xy}
\usepackage{hyperref}
\input xy
\xyoption{all}

\def\from{\colon}

\DeclareMathOperator{\ab}{ab}
\DeclareMathOperator{\rec}{rec}
\DeclareMathOperator{\nr}{nr}
\def\C{\mathbb{C}}

\def\Q{\mathbb{Q}}

\def\Z{\mathbb{Z}}
\def\FF{\mathbb{F}}
\def\OO{\mathcal{O}}

\def\M{\mathcal{M}}

\def\V{V}

\def\gp{\mathfrak{p}}

\def\isom{\cong}

\newcommand{\powerseries}[1]{\llbracket #1 \rrbracket}
\newcommand{\laurentseries}[1]{(\!(#1)\!)}
\newcommand{\injects}{\hookrightarrow}

\DeclareMathOperator{\Gal}{Gal}

\DeclareMathOperator{\End}{End}
\DeclareMathOperator{\Aut}{Aut}

\DeclareMathOperator{\pr}{pr}

\DeclareMathOperator{\Spec}{Spec}

\DeclareMathOperator{\Spa}{Spa}

\DeclareMathOperator{\Hom}{Hom}

\DeclareMathOperator{\Spf}{Spf}
\DeclareMathOperator{\univ}{univ}
\DeclareMathOperator{\GL}{GL}

\DeclareMathOperator{\ad}{ad}

\DeclareMathOperator{\Nilp}{Nilp}

\DeclareMathOperator{\CM}{CM}
\DeclareMathOperator{\Fr}{Fr}

\newcommand{\abs}[1]{\left\lvert #1 \right\rvert}
\newcommand{\class}[1]{\left< #1 \right>}
\newcommand{\set}[1]{\left\{ #1 \right\}}
\newcommand{\norm}[1]{\left\lVert #1 \right\rVert}
\newcommand{\ceil}[1]{\left\lceil #1 \right\rceil}
\newcommand{\floor}[1]{\left\lfloor #1 \right\rfloor}
\newcommand{\tatealgebra}{\class}

\DeclareMathOperator{\Nil}{Nil}

\def\Sets{\mathbf{Sets}}
\def\Vect{\mathbf{Vect}}

\def\bG{\hat{\mathbf{G}}}

\def\F{{K}}
\def\V{{\mathcal{V}}}
\def\bG{\mathbf{G}}
\def\Sets{\mathbf{Sets}}
\def\Mod{\mathbf{Mod}}
\def\Alg{\mathbf{Alg}}
\def\Nilp{\mathbf{Nilp}}
\def\Vect{\mathbf{Vect}}
\def\TopVect{\mathbf{TopVect}}
\def\gp{\mathfrak{p}}
\def\gP{\mathfrak{P}}

\numberwithin{equation}{subsection}
\newtheorem{Theorem}{Theorem}[subsection]
\numberwithin{Theorem}{subsection}
\newtheorem{lemma}[Theorem]{Lemma}

\newtheorem{prop}[Theorem]{Proposition}

\theoremstyle{definition}
\newtheorem{defn}[Theorem]{Definition}

\title{Formal vector spaces over a local field of positive characteristic}
\author{Jared Weinstein}

\begin{document}
\maketitle

\begin{abstract}
Let $\OO_K=\mathbf{F}_q\powerseries{\pi}$ be the ring of power series in one variable over a finite field, and let $K$ be its fraction field.  We introduce the notion of a ``formal $K$-vector space";  this is a certain kind of $K$-vector space object in the category of formal schemes. This concept runs parallel to the established notion of a formal $\OO_K$-module, but in many ways formal $K$-vector spaces are much simpler objects. Our main result concerns the Lubin-Tate tower, which plays a vital role in the local Langlands correspondence for $GL_n(K)$. Let $A_m$ be the complete local ring parametrizing deformations of a fixed formal $\OO_K$-module over the residue field, together with Drinfeld level $\pi^m$ structure. We show that the completion of the union of the $A_m$ has a surprisingly simple description in terms of formal $K$-vector spaces. This description shows that the generic fiber of the Lubin-Tate tower at infinite level carries the structure of a perfectoid space. As an application, we find a family of open affinoid neighborhoods of this perfectoid space whose special fibers are certain remarkable varieties over a finite field which we are able to make completely explicit. It is shown in joint work with Mitya Boyarchenko that the $\ell$-adic cohomology of these varieties realizes the local Langlands correspondence for a certain class of supercuspidal representations of $GL_n(K)$.
\end{abstract}

\section{Introduction}
Let $K$ be a local nonarchimedean field, and let $G_0$ be a one-dimensional formal $\OO_K$-module defined over the algebraic closure of the residue field of $K$.  In~\cite{DrinfeldElliptic} Drinfeld shows that the parameter space of deformations of $G_0$ is an affine formal scheme $\Spf A_0$, where $A_0$ is isomorphic to a power series ring in $n-1$ variables over $\OO_K$, where $n$ is the height of $G_0$.   By considering deformations of $G_0$ equipped with Drinfeld level structures, Drinfeld constructs a tower $A_0\to A_1\to\dots$ of regular local rings.  These rings play a vital role in the local Langlands correspondence for $\GL_n(K)$.  Namely, the direct limit as $m\to \infty$ of the space of vanishing cycles of the formal scheme $\Spf A_m$ is known to realize the local Langlands correspondence (\cite{Boyer}, \cite{HarrisTaylor}).

In this paper we consider the case where $K$ has positive characteristic.  Let $A$ be the completion of $\varinjlim A_m$ with respect to the topology induced by the maximal ideal of $A_0$.  We give a surprisingly simple description of $A$, which allows us to make the actions of $\GL_n(K)$ and its inner twist completely explicit.  Our description is in terms of {\em formal $K$-vector spaces}, which are certain $K$-vector space objects in the category of formal schemes.  For instance, one obtains a formal $K$-vector space from $G_0$ by considering the ``universal cover" $\tilde{G}_0=\varprojlim G_0$, where the transition map is multiplication by a uniformizer of $K$.  In the mixed characteristic setting, a similar construction appears in \cite{FaltingsCoverings}, where it plays a role in the geometry of the period domain associated to $G_0$, and also in the preprint \cite{FarguesFontaine}, where the term ``espace vectoriel formel" is used.

A striking property of formal $K$-vector spaces is that they lift canonically through (topologically) nilpotent thickenings.  In particular we may lift $\tilde{G}_0$ to a formal $K$-vector space $\tilde{G}$ defined over $\Spf\OO_K$.  Our main result (Thm.~\ref{mainthm}) is that $\Spf A$ is the formal subscheme of $\tilde{G}^n$ cut out by a certain determinant condition.

To give the flavor of our result, let us spell out the $n=1$ case.  Let $q$ be the size of the residue field of $K$ and let $\pi$ be a uniformizer, so that $K=\FF_q\laurentseries{\pi}$.  Let $K^{\nr}=\overline{\FF}_q\laurentseries{\pi}$.  Then $A_m=\OO_{K_m}$ is the ring of integers in the field $K_m$ obtained by adjoining to $K^{\nr}$ a primitive root $t_m$ of the $m$th iterate of a Lubin-Tate power series $f(T)$ (for instance $f(T)=\pi T+ T^q$).  We assume that $f(t_m)=t_{m-1}$.  Let $K_\infty$ be the union of the $K_m$;  then $A$ is the ring of integers in the completion of $K_\infty$.  Consider the element $t=\lim_{m\to\infty} t_m^{q^m}$, which converges in $A$.  Then
\begin{equation}
\label{Apowerseries}
A = \overline{\FF}_q\powerseries{t^{1/q^\infty}},
\end{equation}
where the ring on the right is the ring of {\em fractional power series}; {\em i.e.}, it is the $t$-adic completion of the perfect closure of $\overline{\FF}_q[t]$.  The Galois group $\Gal(K_\infty/K^{\nr})$ operates on $A$.  With respect to the (suitably normalized) isomorphism $\Gal(K_\infty/K^{\nr})\isom\OO_K^\times$ of class field theory, we can describe this action as follows:  a power series $\sum_m a_m\pi^m\in \OO_K^\times$ operates on $A$ by sending $t$ to the $\FF_q$-linear power series $\sum_m a_mt^{q^m}$.  A curious feature of this story is that the element $\pi$ lies in $A$, and must be expressible as a fractional power series in $t$ by Eq.~\eqref{Apowerseries} which is $\OO_K^\times$-invariant, but there seems to be no simple formula for it.   

In the general case, we find an isomorphism
\[ A\isom\overline{\FF}_q\powerseries{X_1^{1/q^\infty},\dots,X_n^{1/q^\infty}}.\]  
With respect to this isomorphism, the action of $\GL_n(\OO_K)$ can be made perfectly explicit.  Let $g\in \GL_n(\OO_K)$ have $ij$th entry equal to $\sum_{m\geq 0} a_m^{(ij)}\pi^m$.  Then $g$ sends $X_i$ to 
\[ \sum_{j=1}^n a_m^{(ij)}X_j^{q^m}. \]
Let $J^0$ be the group of automorphisms of $G_0$;  then $J^0$ is the group of units in the algebra of noncommutative power series $\sum_{m\geq 0}a_m\varpi^m$ over $\FF_{q^n}$, where the indeterminate $\varpi$ satisfies $\varpi a=a^q \varpi$ for all $a\in\FF_{q^n}$.  Then the action of $J^0$ on $A$ can also be made explicit:  such a power series sends each $X_i$ to $\sum_{m\geq 0} a_mX_i^{q^m}$.  In contrast, the action of $J^0$ on $A_0=\OO_{\hat{K}^{\nr}}\powerseries{u_1,\dots,u_{n-1}}$ is a complete mystery.  


With such an explicit description of the Lubin-Tate tower at infinite level, there arises the tantalizing hope for a purely local proof of the local Langlands correspondence.  In \S\ref{affinoids} we present some encouraging results in this direction.  Let $\mathcal{M}^{\ad}_{\overline{\eta}}$ be the adic geometric generic fiber of $\Spf A$. Then $\mathcal{M}^{\ad}_{\overline{\eta}}$ is a perfectoid space in the sense of \cite{Sch}.  We construct a nested family $\mathcal{Z}_1,\mathcal{Z}_2,\dots$ of open affinoid subspaces of $\mathcal{M}^{\ad}_{\overline{\eta}}$.  For each $m$ we compute the special fiber $\mathcal{Z}_m$ and observe that it is the perfect closure of a nonsingular affine variety $X_m$ which descends to the prime field.  The actions of the relevant groups on $X_m$ remain perfectly explicit.

Therefore it becomes urgent to compute the $\ell$-adic cohomology of the varieties $X_m$.  This is the content of joint work with Mitya Boyarchenko, which computes the cohomology of $X_m$ in every degree.  See \cite{BW} for the details.   The $X_m$ are quite remarkable.  One interesting property is that the $q^n$th power Frobenius map acts as the scalar $(-1)^{i-1}q^{ni/2}$ on the degree $i$ compactly supported cohomology of $X_m$.  Another property is that there is a large group $U$ of automorphisms of $X_m$ such that the representation of $U$ on $\bigoplus_i H^i_c(X_m,\overline{\Q}_\ell)$ decomposes into irreducibles {\em with multiplicity one}. \cite{BW} also confirms that the cohomology in the middle degree accounts for the local Langlands correspondence for a class of supercuspidals of $\GL_n(K)$.  This class consists of the supercuspidal representations of $\GL_n(K)$ whose Weil parameter is induced from a one-dimensional character of the degree $n$ unramified extension of $K$ of conductor $\pi^{m+1}$, so long as that character satisfies a certain primitivity condition. 

We find it likely that there are enough special affinoids in $\mathcal{M}^{\ad}_{\overline{\eta}}$ to account for the entire supercuspidal spectrum of $\GL_n(K)$, and remain hopeful that their study could lead to a purely local proof of the existence of the local Langlands correspondence.  The method of \cite{BW} produces {\em some} correspondence between supercuspidals and Weil parameters;  this is shown manually to agree with the local Langlands correspondence, which happens to be known explicitly in these cases by work of Kazhdan and Henniart.   For general supercuspidals it might be possible to check the correspondence afforded by the special affinoids against the purely local characterization of the local Langlands correspondence given in \cite{ScholzeLLC}.

A detailed study of the Lubin-Tate tower in equal characteristic has been carried out by Genestier and Lafforgue, in the second half of the book \cite{Fargues}.  There the authors construct an isomorphism between two formal schemes:  one which is the inverse limit of formal models of the layers of the Lubin-Tate tower, and one which is the inverse limit of formal models of the layers of the Drinfeld tower.  These formal schemes are rather more intricate than our $\Spf A$--indeed, the special fiber of the Genestier-Lafforgue model is an infinite configuration of infinitely generated schemes, whereas the special fiber of $\Spf A$ is just a point.  However, one notices that a significant role is played by the determinant morphism in both constructions.

Every result in this paper has an analogue in the mixed characteristic case.   All Rapoport-Zink spaces at infinite level are perfectoid spaces which can be described in terms of ``formal linear algebra":  this is the subject of forthcoming work with Peter Scholze.  The results of \S\ref{affinoids} hold in the mixed characteristic case as well--the varieties $X_m$ should be precisely the same--but the computations slightly more difficult because the operations in the relevant formal vector space $\tilde{G}_0$ cannot be made quite so explicit.

We thank Mitya Boyarchenko for pointing out many errors in an earlier draft of this paper.

\subsection{Notation and basic definitions}
\begin{defn} A topological ring $R$ is {\em adic} if there exists an ideal $I\subset R$ such that $R$ is separated and complete for the $I$-adic topology.  Such an $I$ is called an {\em ideal of definition} for $R$.  
\end{defn}

We will be working with the following categories (in addition to the obvious $\Sets$):
\begin{itemize}
\item $\Vect_k$, vector spaces over a field $k$
\item $\TopVect_k$, topological vector spaces over a topological field $k$
\item $\Mod_R$, modules over a ring $R$
\item $\Alg_R$, $R$-algebras
\item $\Nilp_R$, $R$-algebras in which $I$ is nilpotent (where $R$ is an adic ring with ideal of definition $I$)
\end{itemize}

We often consider covariant functors $\mathcal{F}\from\Alg_R\to \Sets$, where $R$ is some ring.  Recall that such a functor is {\em pro-representable} by an affine formal scheme $\Spf A$ lying over $\Spec R$ if there is an ideal of definition $I$ of $A$ such that $\mathcal{F}(S)=\varinjlim \Hom_{\Alg_R}(A/I^n,S)$ for all $R$-algebras $S$.  In this scenario we will often confuse the functor $\mathcal{F}$ with the formal scheme which pro-represents it.

More generally, suppose $R$ is an adic ring with ideal of definition $I$.  A covariant functor $\mathcal{F}\from \Nilp_R\to \Sets$ is pro-representable by an affine formal scheme $\Spf A$ if there is an ideal of definition $J$ of $A$ containing $I$ such that $\mathcal{F}(S)=\varinjlim \Hom_{\Nilp_R}(A/J^n,S)$ for all objects $S$ of $\Nilp_R$.  In that case we have a morphism of formal schemes $\Spf A\to \Spf R$.


\section{Formal vector spaces}

\subsection{Formal $\FF_q$-vector spaces and coordinate modules}

For any $\FF_q$-algebra $R$, write $\tau\from R\to R$ for the $q$th power Frobenius endomorphism.  Let $\Nil(R)$ be the nilpotent radical of $R$.  The functor $R\mapsto \Nil(R)$ is pro-representable by the formal scheme $\Spf\FF_q\powerseries{X}$.   The functor $\Nil^\flat$ which assigns to $R$ the projective limit $\Nil^\flat(R)=\varprojlim_\tau \Nil(R)$ is pro-representable by the formal spectrum of the $X$-adic completion of $\varinjlim_\tau \FF_q\powerseries{X}$, which we call $\FF_q\powerseries{X^{1/q^\infty}}$.  Explicitly, $\FF_q\powerseries{X^{1/q^\infty}}$ is the ring of fractional power series of the form $\sum_\alpha c_\alpha X^\alpha$, where $\alpha\in\Z[1/q]_{\geq 0}$, subject to the constraint that for every $N>0$, there are only finitely many nonzero coefficients $c_\alpha$ with $\alpha<N$.

For any perfect field $k$ containing $\FF_q$, and any $d\geq 1$, let $\bG_{a,k}^{\flat,\oplus d}\from\Alg_k\to \TopVect_{\FF_q}$ be the functor defined by
\[ \bG_{a,k}^{\flat,\oplus d}(R)=\varprojlim_{\tau}\Nil(R)^{\oplus d}. \]
(The topology on $\bG_{a,k}^{\flat,\oplus d}(R)$ is the inverse limit topology.)   Then $\bG_{a,k}^{\flat,\oplus d}$ is pro-representable by a formal scheme, namely $\Spf k\powerseries{X_1^{1/q^\infty},\dots,X_n^{1/q^\infty}}$.   Thus $\mathbf{G}_{a,k}^{\flat,\oplus d}$ is an $\FF_q$-vector space object in the category of formal schemes over $\Spec k$.

\begin{defn}\label{Fqvs} Let $k$ be a perfect field containing $\FF_q$.  A {\em formal $\FF_q$-vector space} of dimension $d$ over $k$ is a
functor $\V\from \Alg_k\to  \TopVect_{\FF_q}$, such that there is an isomorphism $\V\approx\mathbf{G}_{a,k}^{\flat,\oplus d}$.  Morphisms between formal $\FF_q$-vector spaces are natural transformations between functors.
\end{defn}

The $q$th power Frobenius morphism $\tau$ induces an $\FF_q$-linear automorphism of $\bG_{a,k}^{\flat}$.  The ring $\End_{\FF_q}\bG_{a,k}^{\flat}$ of $\FF_q$-linear endomorphisms of $\bG_{a,k}^{\flat}$ is the ring of noncommutative formal Laurent series $k\laurentseries{\tau}$, where for $\alpha\in k$, $\tau\alpha=\alpha^q\tau$.  Note that $k\laurentseries{\tau}$ is naturally a topological ring, and for any $k$-algebra $R$, the action of $k\laurentseries{\tau}$ on the topological $\FF_q$-vector space $\bG_{a,k}^{\flat}(R)$ is continuous.

The following notion is adapted from the definition of ``module de coordonn\'ees" appearing in \cite{Genestier}, Ch. I.  

\begin{defn}\label{coordinatemodule}  Let $\V$ be a formal $\FF_q$-vector space over $k$.  The {\em coordinate module} of $\V$ is the $k$-vector space $M(\V)=\Hom(\V,\bG_{a,k}^{\flat})$.  (The $\Hom$ is in the category of formal $\FF_q$-vector spaces over $k$.)
\end{defn}

By the remark immediately preceding the definition, $M(\V)$ is naturally a left $k\laurentseries{\tau}$-module which is free of rank $d$.

\begin{prop} \label{antiequiv} The functor $\V\mapsto M(\V)$ is an anti-equivalence from the category of formal $\FF_q$-vector spaces over $k$ to the category of free left $k\laurentseries{\tau}$-modules of finite rank.
\end{prop}

\begin{proof} We describe the inverse functor.  If $M$ is a free $k\laurentseries{\tau}$-module of finite rank, let $\V\from \Alg_k\to \TopVect_{\FF_q}$ be the functor which inputs a $k$-algebra $R$ and outputs the set of continuous homomorphisms $e\from M\to \Nil(R)$ which satisfy $e(\tau m)=e(m)^q$ for all $m\in M$.  (Here $M$ has the topology induced from $k\laurentseries{\tau}$, while $\Nil(R)$ has the discrete topology.)  Then $\V$ is pro-representable by the desired formal $\FF_q$-vector space.
\end{proof}

\subsection{Formal $K$-vector spaces}
Now suppose $K$ is a local field of positive characteristic.  Write $K=\FF_q\laurentseries{\pi}$.

\begin{defn} Let $k$ be a perfect field containing $\FF_q$.  A {\em formal $K$-vector space} of dimension $d$ over $k$ is a functor $\V\from \Alg_k\to \TopVect_K$, whose composition with the forgetful functor $\TopVect_K\to\TopVect_{\FF_q}$ is a formal $\FF_q$-vector space of dimension $d$ in the sense of Defn.~\ref{Fqvs}.
\end{defn}

In the case that $\V$ has underlying formal $\FF_q$-vector space $\bG^{\flat}_{a,k}$, then giving the $K$-vector space structure of $\V$ is one and the same as giving a continuous homomorphism $K\to \End\bG_{a,k}^{\flat}=k\laurentseries{\tau}$, which is in turn the same as specifying the image of $\pi\in K$, which must be an invertible, topologically nilpotent element of $k\laurentseries{\tau}$.  One sees the analogy with Drinfeld modules:  a one-dimensional formal $K$-vector space is the local field version of a Drinfeld module, with the additional twist that the action of $\pi$ has been rendered invertible.

If $\V$ is a formal $K$-vector space over $k$, then $M(\V)$ naturally becomes a module over $K\hat{\otimes}_{\FF_q} k\laurentseries{\tau}$.  That is, $M(\V)$ is a topological $K\hat{\otimes}_{\FF_q} k$-vector space together with a topologically nilpotent $1\otimes \tau$-semilinear operator $F\from M(\V)\to M(\V)$.   Here we have abused notation slightly in writing $\tau\from k\to k$ as the $q$th power Frobenius endomorphism of $k$.  Note that $K\hat{\otimes}_{\FF_q} k=k\laurentseries{\pi}$ is a field.

It is easy to see that $M(\V)$ has finite dimension as a vector space over $k\laurentseries{\pi}$;  we call this dimension the {\em height} of $\V$.  For instance, if $\V$ is a one-dimensional formal $K$-vector space corresponding to a continuous homomorphism $K\to k\laurentseries{\tau}$, then the image of $\pi$ in $k\laurentseries{\tau}$, being topologically nilpotent, must have the form $a_n\tau^n+a_{n+1}\tau^{n+1}$ for some $n\geq 1$, $a_n\neq 0$.  Then $n$ is the height of $\V$.  

\begin{defn}  Let $k$ be a perfect field containing $\FF_q$.   A {\em $K$-isocrystal} over $k$ is a finite-dimensional vector space $M$ over $K\hat{\otimes}_{\FF_q} k = k\laurentseries{\pi}$ together with a $1\otimes\tau$-semilinear automorphism $F\from M\to M$.  A $K$-isocrystal $M$ is {\em nilpotent} if $F$ is topologically nilpotent.
\end{defn}

Thus for a formal $K$-vector space $\V$, the coordinate module $M(\V)$ is a nilpotent $K$-isocrystal.

From Prop.~\ref{antiequiv} we easily deduce the following proposition.
\begin{prop}  \label{antiequivK}  The functor $\V\mapsto M(\V)$ is an anti-equivalence from the category of formal $K$-vector spaces over $k$ to the category of nilpotent $K$-isocrystals over $k$.
\end{prop}

\begin{proof} Once again, we indicate the inverse functor.  If $M$ is a nilpotent $K$-isocrystal, define a functor $\V\from \Alg_k\to \TopVect_K$ by
\[ \V(R) = \Hom_{F,\tau}(M,\Nil(R)),\]
meaning the $K$-vector space of continuous $\FF_q$-linear maps $\lambda\from M\to \Nil(R)$ satisfying $\lambda(F(m))=\lambda(m)^q$, for $m\in M$.  Then $\V$ is the formal $K$-vector space corresponding to $M$.
\end{proof}

Given a $K$-isocrystal $M$, let $M^*$ be the space of continuous $\FF_q$-linear maps $\lambda\from M\to \FF_q$.  Then $M^*$ becomes a $k\laurentseries{\pi}$-vector space through $f\lambda(m)=\lambda(fm)$, for $f\in k\laurentseries{\pi}$, $\lambda\in M^*$, $m\in M$.  Also define a $1\otimes\tau$-linear map $F\from M^*\to M^*$ via $F\lambda(m)=\lambda(F^{-1}(m))$.   Then $M^*$ is also an isocrystal;  in particular it is a finite-dimensional $k\laurentseries{\pi}$-vector space.
Note that $F$ is topologically nilpotent on $M$ if and only if $F^{-1}$ is topologically nilpotent on $M^*$.

\begin{lemma} \label{Mstar}  For a formal $K$-vector space $\V$ over $k$ with associated $K$-isocrystal $M$, we have a natural isomorphism of topological $K$-vector spaces $\V(R)\to (M^*\hat{\otimes}
\Nil(R))^{F\otimes\tau}$ for every $R$.  Here $(M^*\hat{\otimes}\Nil(R))^{F\otimes\tau}$ is the space of vectors in $M^*\hat{\otimes}\Nil(R)$ which are invariant under $F\otimes\tau$.
\end{lemma}

\begin{proof}  By the proof of Prop.~\ref{antiequivK}, we have $\V(R)=\Hom_{F,\tau}(M,\Nil(R))$, and this is the same as $(M^*\hat{\otimes}\Nil(R))^{F\otimes\tau}$.
\end{proof}


\subsection{Some multilinear algebra}
\label{multilinear}
Given two formal $K$-vector spaces $\V$ and $\mathcal{W}$ over $k$, a natural transformation of functors $\delta\from\V^n\to \mathcal{W}$ is {\em $K$-alternating} if for each $k$-algebra $R$, the map $\mu(R)\from \V(R)^n\to \mathcal{W}(R)$ is an alternating $K$-multilinear map.  Similarly, if $M$ and $N$ are two $K$-isocrystals, a map $f\from M^n\to N$ is a $K$-alternating map of isocrystals if it is $K$-alternating as a map between $K$-vector spaces and if it also satisfies $f(F(m_1),\dots,F(m_n))=F(f(m_1,\dots,m_n))$.

For a formal $K$-vector space $\V$ over $k$, let $M(\V)$ be the corresponding nilpotent $K$-isocrystal, and let $M^*(\V)=M(\V)^*$.
\begin{Theorem} \label{detexistence}
$K$-alternating maps $\V^n\to \mathcal{W}$ are in bijection with $K$-alternating maps of isocrystals $M^*(\V)^n\to M^*(\mathcal{W})$.
\end{Theorem}

\begin{proof}
Given a $K$-alternating map of isocrystals $\mu\from M^*(\V)^n\to M^*(\mathcal{W})$, one gets an obvious $K$-alternating map
\begin{equation}
\label{lambdaR}
\lambda_R\from\left((M^*(\V)\hat{\otimes}\Nil(R))^{F\otimes\tau}\right)^n \to (M^*(\mathcal{W})\hat{\otimes}\Nil(R))^{F\otimes\tau}
\end{equation}
which is natural in $R$;  thus by Lemma~\ref{Mstar} we have a $K$-alternating map $\lambda\from \V^n\to W$.

Conversely, suppose a $K$-alternating map $\lambda\from \V^n\to \mathcal{W}$ is given.  To construct an alternating map $M^*(\V)^n\to M^*(\mathcal{W})$, suppose we are given elements $e_1,\dots,e_n\in M^*(\V)$.  Let $R=k[x_1^{1/q^\infty},\dots,x_n^{1/q^\infty}]/(x_1^q,\dots,x_n^q)$.  Let $\pr\from R\to k$ be the projection onto the coefficient of $x_1\cdots x_n$.  For $j=1,\dots,n$, let $\alpha_j=\sum_{i\geq 0} F^{-i}(e_j)\otimes x_j^{q^{-i}}$.  Note that the sum is convergent because $F^{-1}$ is topologically nilpotent on $M^*(\V)$.  Clearly $\alpha_j$ is invariant under $F\otimes \tau$, so that $\alpha_j$ may be identified with an element of $\mathcal{V}(R)$ (Lemma~\ref{Mstar}).  We get the element $\lambda(\alpha_1,\dots,\alpha_n)\in \mathcal{W}(R)\subset M^*(\mathcal{W})\hat{\otimes} \Nil(R)$.  Let
\[ \mu(e_1,\dots,e_n)=(1\otimes\pr)\left(\lambda(\alpha_1,\dots,\alpha_n)\right)\in M^*(\mathcal{W}). \]
It is not hard to see that $\mu\from M^*(\V)^n\to M^*(\mathcal{W})$ is a $K$-alternating map of isocrystals, and that $\lambda\mapsto \mu$ is a double-sided inverse to the map $\mu\mapsto \lambda$ of the previous paragraph.
\end{proof}

It is a simple matter to see that exterior powers exist in the category of $K$-isocrystals, and that every exterior power of a nilpotent $K$-isocrystal is also nilpotent. 
If $\V$ is a formal $K$-vector space of height $n$, let $\wedge \V$ be the formal $K$-vector space associated to the top exterior power $\wedge^n M(\V)$.  Then $\wedge \V$ has height one. Thm.~\ref{detexistence} shows there is a $K$-alternating map $\delta\from \V^n\to \wedge\V$ which is universal among all $K$-alternating maps out of $\V^n$.

\subsection{Formal vector spaces over adic rings}
\label{FormalVSoverAdic}

Let $A$ be a complete local $\OO_K$-algebra with maximal ideal $I$, such that $k=A/I$ is a perfect field containing $\FF_q$.   Let $\V_0$ be an $n$-dimensional formal $K$-vector space over $k$.  We define a functor $\V\from \Nilp_{A}\to\TopVect_K$ by $\V(R)=\V_0(R/I)$.

\begin{lemma} \label{FormalVSadic}  $\V$ is pro-representable by a formal scheme which is isomorphic to $\Spf A\powerseries{X_1^{1/q^\infty},\dots,X_n^{1/q^\infty}}$.
\end{lemma}

\begin{proof}  After choosing coordinates on $\V_0$, we have $\V_0(R/I)=\varprojlim_\tau \Nil(R/I)^n$.  But this inverse limit is in functorial bijection with $\varprojlim_\tau\Nil(R)^n$, via the Teichm\"uller lift.  Thus $\V(R)=\varprojlim_\tau\Nil(R)^n$ as sets, so that $\V$ is pro-representable by $\Spf A\powerseries{X_1^{1/q^\infty},\dots,X_n^{1/q^\infty}}$ as claimed.
\end{proof}

Let us define a formal $K$-vector space over $A$ to be a functor $\V\from \Nilp_A\to\TopVect_K$ which is pro-representable by $\Spf A\powerseries{X_1^{1/q^\infty},\dots,X_n^{1/q^\infty}}$.  We have thus constructed a functor
$\V_0\to \V$ from formal $K$-vector spaces over $k$ to formal $K$-vector spaces over $A$.  We will refer to $\V$ as the lift of $\V_0$ to $A$.

If $\delta_0\from \V_0^n\to \wedge\V_0$ is the determinant map from \S\ref{multilinear}, then it is easy to see that $\delta_0$ lifts to a $K$-alternating map $\delta\from \V^n\to\wedge V$ of formal $K$-vector spaces over $A$, and that $\delta$ is the universal $K$-alternating map out of $\V^n$.

\section{Formal $\OO_K$-modules}

Let $A$ be an adic $\OO_K$-algebra.  A {\em formal $\OO_K$-module} $G$ over $A$ is an $\OO_K$-module object in the category of formal schemes over $\Spf A$.  It is required that $G$ be isomorphic\footnote{One ought to work with a more general notion of formal $\OO_K$-module, where the isomorphism $G\approx \Spf A\powerseries{X_1,\dots,X_d}$ is only required to exist locally on $\Spf A$.  As $\Spf A$ will be a single point for all the rings relevant to this paper, we shall not worry about this issue.} to $\Spf A\powerseries{X_1,\dots,X_d}$ for some $d$ (the dimension of $G$), and that the action of $\OO_K$ on the Lie algebra of $G$ agree with the structural homomorphism $\OO_K\to A$.   Once a choice of coordinates $G\approx \Spf A\powerseries{X_1,\dots,X_d}$ is made, this definition agrees with the usual definition in terms of power series.


A formal $\OO_K$-module $G$ over $A$ pro-represents a functor $\Nilp_A\to \Mod_{\OO_K}$, which we write as $R\mapsto G(R)$.  Once a choice of coordinates on $G$ is made, the underlying functor $G\from \Alg_A\to \Sets$ is identified with $R\mapsto \Nil(R)^d$.  Since $p=0$ in $\OO_K$, there is no loss of generality in assuming that the underlying $\FF_q$-vector space of $G$ is the additive group.  We will always make this assumption.

\subsection{The universal cover}
\label{univcover}

\begin{defn}  Given a formal $\OO_K$-module $G$ over an adic $\OO_K$-algebra $A$, the {\em universal cover} of $G$ is the functor $\tilde{G}\from\Nilp_A\to\TopVect_K$ defined by
\[ \tilde{G}(R)=\varprojlim_\pi G(R) \]
\end{defn}

Note that $\tilde{G}$ is pro-representable by a formal scheme over $\Spf A$.  Indeed,
$G$ is pro-representable by a formal scheme isomorphic to $\Spf A\powerseries{X_1,\dots,X_n}$.  Thus $\tilde{G}$ is representable by the formal spectrum of $\varinjlim A\powerseries{X_1,\dots,X_n}$, where the limit is taken along the endomorphisms of $A\powerseries{X_1,\dots,X_n}$ corresponding to multiplication by $\pi$ in $G$.

\begin{prop} \label{IsKvs} Let $G$ be a $\pi$-divisible formal $\OO_K$-module of height $n$ over a perfect field $k$ containing the residue field $\FF_q$ of $\OO_K$.  Then $\tilde{G}$ is a formal $K$-vector space over $k$ of height $n$.
\end{prop}

\begin{proof}  The isogeny $\pi\from G\to G$ factors as $\pi=F^nu$, where $F\from G\to G^{(q)}$ is the Frobenius isogeny and $u\from G^{(q^{-n})}\to G$ is an isomorphism.  For $m\geq 1$ let us write the composite $G^{(q^{-mn})}\to G^{(q^{-(m-1)n})} \to \dots \to G$ simply as $u^m$.  We have an isomorphism of functors $\Alg_k\to \TopVect_K$:
\begin{equation}
\label{Glimit}
\varprojlim_{\pi} G \to \varprojlim_{F^n} G^{(q^{-mn})}
\end{equation}
given by
\[ (x_1,x_2,\dots)\mapsto (u(x_1),u^2(x_2),\dots). \]
Therefore $\tilde{G}$ is isomorphic to $\varprojlim_{F} G^{(q^{-m})}$.  For a $k$-algebra $R$ we have $\tilde{G}(R)\isom \varprojlim_F G^{(q^{-m})}(R)\approx \varprojlim_F \Nil(R)$ as topological $\FF_q$-vector spaces.  The action of $\pi$ on $\tilde{G}(R)$ agrees with that of $F^n$ up to an invertible operator.  This means that $\tilde{G}$ is a formal $K$-vector space of height $n$.
\end{proof}

For a formal $\OO_K$-module $G$ over $A$, let $G[\pi^m]\from \Nilp_A\to \Mod_{\OO_K}$ be the functor defined by $G[\pi^m](R)=\set{x\in G(R)\vert \pi^mx=0}$.  Then $\varprojlim G[\pi^n](R)$ is a torsion-free $\OO_K$-module which admits an obvious injection into $\tilde{G}(R)=\varprojlim_\pi G(R)$.  Since $\tilde{G}(R)$ is a $K$-vector space, we have a map $\varprojlim G[\pi^n](R)\otimes_{\OO_K} K\to \tilde{G}(R)$.
\begin{lemma}
\label{Tatemod}  If $\pi$ is nilpotent in $R$, then
$\varprojlim G[\pi^n](R)\otimes_{\OO_K} K\to \tilde{G}$ is an isomorphism.
\end{lemma}
\begin{proof} This follows from the fact that $\varprojlim G[\pi^n](R)$ does not have any $\pi$-torsion, and from the fact that every element in $G(R)$ is $\pi^m$-torsion for some $m$, because $\pi$ is nilpotent in $R$.
\end{proof}


\begin{prop}\label{reductionmap}  Let $A$ be an adic $\OO_K$-algebra and let $I\subset A$ be an ideal of definition which contains $\pi$.
 Let $G$ be a formal $\OO_K$-module over $A$.  Consider the functor $\tilde{G}\from\Nilp_A\to \TopVect_K$.
 \begin{enumerate}
 \item For every object $R$ of $\Nilp_A$, the natural reduction map $\tilde{G}(R)\to \tilde{G}(R/I)$ is an isomorphism.
 \item If we further assume that $A/I$ is a perfect field, then $\tilde{G}$ is a formal $K$-vector space over $A$.
 \end{enumerate}
\end{prop}


\begin{proof}  Choose a coordinate on $G$, so that the underlying $\FF_q$-vector space of $\tilde{G}(R)$ may be identified with $\Nil(R)^{\oplus d}$.  Let $I_R$ be the extension of $I$ to $R$, so that $I_R$ is nilpotent. If $(x_1,x_2,\dots)\in \tilde{G}(R)$ lies in the kernel of $\tilde{R}(A)\to\tilde{G}(R/I)$, then each $x_i$ lies in $I_R^{\oplus d}$.  But the power series giving mutliplication by $\pi$ in $G$ has $\pi\in I$ as its linear term, so it carries $(I_R^m)^{\oplus d}$ onto $(I_R^{m+1})^{\oplus d}$.  It follows quickly that each $x_i$ lies in $\bigcap_m (I_R^m)^{\oplus d}=0$.

We show the reduction map is surjective using the standard ``Teichm\"uller lift".  Suppose $(x_1,x_2,\dots)\in\tilde{G}(R/I)$.   Since $I$ is nilpotent in $R$, we may lift each $x_i$ to an element $y_i\in \tilde{G}(R)$.  These will satisfy $\pi y_i - y_{i-1}\in \ker(G(R)\to G(R/I))$.   It is easy to check that the sequence $y_i$, $\pi y_{i+1}$, $\pi^2 y_{i+2}$ must stabilize to an element $z_i\in G(R)$;  then $(z_1,z_2,\dots)\in \tilde{G}(R)$ lifts $(x_1,x_2,\dots)\in\tilde{G}(R/I)$.  This settles part (1).

Part (2) follows from part (1) together with the discussion in \S\ref{FormalVSoverAdic}:  we have that $\tilde{G}_0$ is a formal $K$-vector space over $k$, and part (1) shows that $\tilde{G}$ is its lift to $A$.
\end{proof}

The first part of the lemma shows that the functor $\tilde{G}\from \Nilp_A\to \TopVect_K$ only depends on $\tilde{G}\otimes_A A/I$, in a functorial sense.  That is, there is a functor
\begin{eqnarray*}
\set{\text{Formal $\OO_K$-modules over $A/I$}} &\to& \set{\text{Formal schemes over $\Spf A$}} \\
G_0 &\mapsto & \tilde{G},
\end{eqnarray*}
where $\tilde{G}$ pro-represents the functor $R\mapsto \tilde{G}_0(R/I)$ for any object $R$ of $\Nilp_A$.  Then if $G'$ is a lift of $G_0$ to a formal $\OO_K$-module over $R$, then we have a canonical isomorphism of functors $\tilde{G}'\isom\tilde{G}$.



\section{The Lubin-Tate deformation problem at infinite level}

\subsection{Some local deformation rings}

We recall the setup:  $K=\FF_q\laurentseries{\pi}$, $k$ is a perfect field containing $\FF_q$, and $L=K\hat{\otimes}k=k\laurentseries{\pi}$.  Let $G_0$ be a formal $\OO_K$-module over $k$ of dimension 1 and height $n$.

Let $\mathcal{C}$ be the category of complete local Noetherian $\OO_L$-algebras with residue field $k$.  The functor $\mathcal{C}\to\Sets$ which assigns to $R$ the set of deformations of $G_0$ to $R$ is representable by a ring $A_0$ which is non-canonically isomorphic to $\OO_L\powerseries{u_1,\dots,u_{n-1}}$.  Here a {\em deformation} of $G_0$ to $R$ is a pair $(G,\iota)$, where $G$ is a formal $\OO_K$-module and $\iota\from G_0\to G\otimes_R k$ is an isomorphism.

Let $G$ be a deformation of $G_0$ to $R$.  For $m\geq 1$, we have the notion of a Drinfeld level structure for $G[\pi^m]$ over $R$.  In the case that $G$ is one-dimensional, we may choose a coordinate $X$ on $G$;  then for
$\alpha\in\OO_K$ we write $[\alpha]_G(X)=\alpha X + \dots\in R\powerseries{X}$ for the power series determining multiplication by $\alpha$ in $G$.   A {\em Drinfeld level $\pi^m$ structure} on $G$ is an $\OO_K$-module homomorphism $\phi\from (\OO_K/\pi^m)^n\to G[\pi^m](R)$ for which the product $\prod_{v\in (\pi^{m-1}\OO_K/\pi^m)^n} (T-\phi(v))$ is divisible by $[\pi]_G(T)$ in $R\powerseries{T}$.  The $n$-tuple of images of the standard basis elements under $\phi$ will be called a {\em Drinfeld basis} for $G[\pi^m]$ over $R$.

 The functor $\mathcal{C}\to\Sets$ which assigns to $R$ the set of deformations $G$ of $G_0$ to $R$ equipped with a Drinfeld $\pi^m$ level structure is representable by a local $\OO_K$-algebra $A_m$.  If $G^{\univ}$ is the universal deformation over $A_0$, write $X_1^{(m)},\dots,X_n^{(m)}\in A_m$ for the universal Drinfeld basis for $G^{\univ}[\pi^m]$ over $A_m$ (but note that this requires choosing a coordinate on $G^{\univ}$).  Drinfeld shows that $A_m$ is a regular local ring, with parameters $X_1^{(m)},\dots,X_n^{(m)}$.  In our equal characteristic case, $A_m$ is a regular local ring containing its residue field $k$, so that $A_m$ is isomorphic to a power series ring over $k$ in $n$variables corresponding to the $X_i^{(m)}$.  We abuse notation slightly by writing:
\begin{equation}
\label{Ampowerseries}
A_m = k\powerseries{X_1^{(m)},\dots,X_n^{(m)}}
\end{equation}

Let $I$ be the maximal ideal of $A_1$, so that $I=\left(X_1^{(1)},\dots,X_n^{(1)}\right)$.    Considering $[\pi]_{G^{\univ}}(T)$ as a power series with coefficients in $A_1$, we have that
\begin{equation}
\label{piGmodI}
[\pi]_{G^{\univ}}(T)\equiv [\pi]_{G_0}(T)\pmod{I},
\end{equation}
because $G^{\univ}\otimes_{A_1}A_1/I = G_0$.

We will abuse notation by using the same symbol $I$ for the extension of this ideal to the rings $A_m$.  The following proposition computes this extension for all $m$.  Before stating it, we introduce a bit of notation.  Whenever $J$ is an ideal in an $\FF_q$-algebra, we write $J^{[q^m]}$ for the ideal generated by $q^m$th powers of elements of $J$.

\begin{prop}
\label{Iext}
Let $m\geq 1$.  We have
\[ A_m/I = k\left[ X_1^{(m)}, \dots, X_n^{(m)} \right]/\left(X_1^{(m)},\dots,X_n^{(m)}\right)^{[q^{(m-1)n}]}.\]
\end{prop}
\begin{proof}
Let $R$ be an object of $\mathcal{C}$.  A continuous homomorphism $A_m\to R$ corresponds to a deformation $G$ of $G_0$ to $R$ together with an $n$-tuple $x_1,\dots,x_n\in R$ constituting a Drinfeld basis for $G[\pi^m](R)$.   We claim that this homomorphism factors through $A_m/I\to R$ if and only if $\pi=0$ in $R$, $G=G_0\otimes_k R$, and $[\pi^{m-1}]_{G_0}(x_i)=0$ for $i=1,\dots,n$.  (The last condition means that the $x_i$ actually lie in $G[\pi^{m-1}](R)$.) One direction is easy, since these conditions hold in the case of $R=A_m/I$ together with its universal deformation and level structure.  For the other direction, let $A_m\to R$ be the homomorphism carrying $G^{\univ}$ onto $G$ and $X_i^{(m)}$ onto $x_i$.  This homomorphism carries $X_i^{(1)}$ to $0$, hence it factors through $A_m/I$.

Now suppose $\pi=0$ in $R$, and $x_1,\dots,x_n$ are any elements of $G_0[\pi^{m-1}](R)$.  We claim that the tuple $x_1,\dots,x_n$ is a Drinfeld basis for $G_0[\pi^m](R)$.  This comes down to the claim that $T^{q^n}$ is divisible by $[\pi]_{G_0}(T)$ in $R\powerseries{T}$.  But this is certainly true because $G_0$ has height $n$.

Therefore $A_m/I$ has the following moduli interpretation:  for any local Noetherian $k$-algebra $R$, $\Hom(A_m/I,R)$ is in bijection with the set of $n$-tuples $x_1,\dots,x_n\in G_0[\pi^{m-1}](R)$.  Since $[\pi]_{G_0}(T)$ equals $T^{q^n}$ times a unit in $k\powerseries{T}$, the proposition follows.
\end{proof}

Finally, we include a useful lemma which compares the ideal $I$ to the maximal ideals of the $A_m$.

\begin{lemma} \label{idealinclusions} Let $m\geq 1$, and let $I_m$ be the maximal ideal of $A_m$.  Then we have the following inclusions of ideals in $A_m$:
\[ I_m^{nq^{(m-1)n}}\subset I_m^{[q^{({m-1})n)}]} \subset I \subset I_m^{q^{(m-1)n}}. \]
\end{lemma}

\begin{proof} The inclusion $I_m^{nq^{(m-1)n}}\subset I_m^{[q^{({m-1})n)}]}$ follows easily from the fact that $I_m$ is generated by $n$ elements, namely $X_1^{(m)},\dots,X_n^{(m)}$.   The inclusion $I_m^{[q^{({m-1})n)}]} \subset I$ follows from Prop. \ref{Iext}.  For the inclusion $I \subset I_m^{q^{(m-1)n}}$, recall that $I$ is generated by $X_1^{(1)},\dots,X_n^{(1)}$.  We have
\[ X_i^{(1)}=[\pi^{m-1}]_{G^{\univ}}\left(X_i^{(m)}\right).\]
The formal $\OO_K$-module $G_0$ has height $n$, so the power series $[\pi]_{G_0}(T)$ has no terms of degree less than $q^n$.  
Since $G^{\univ}$ is a deformation of $G_0$, all terms of degree less than $q^n$ in $[\pi]_{G^{\univ}}(T)$ are contained in the maximal ideal of $A_0$, which is in turn contained in $I$.  
It follows that all terms of $[\pi^{m-1}]_{G^{\univ}}(T)$ of degree less than $q^{(m-1)n}$ are contained in $I$.  
This shows that $I\subset II_m+I_m^{q^{(m-1)n}}$.  
Iterating this relation shows that $I\subset II_m^r+I_m^{q^{(m-1)n}}$ for all $r\geq 1$, and therefore $I\subset I_m^{q^{(m-1)n}}$.
\end{proof}

\subsection{The deformation functor at infinite level}

Let $A_\infty=\varinjlim A_m$, and let $A$ be the $I$-adic completion of $A_\infty$.  Define the formal scheme $\mathcal{M}_{G_0}$ by
\[ \mathcal{M}_{G_0}=\Spf A.\]
Our goal is to give a description of $\mathcal{M}_{G_0}$ in terms of formal $K$-vector spaces.  This description will have nothing to do with formal $\OO_K$-modules or Drinfeld level structures.  We start by showing that $I$-adically completing $A_\infty$ does not introduce nilpotent elements.

\begin{prop} \label{Areduced} $A$ is reduced.
\end{prop}

\begin{proof}  Suppose that $a\in A$ has $a^2=0$.  Let $a_1,a_2,\dots$ be a sequence of elements of $A_\infty$ which converge to $a$.   Let $N\geq 1$ be arbitrary.  Since $a_i^2\to 0$ $I$-adically, there exists $i_0\geq 1$ such that for all $i\geq i_0$ we have $a_i^2\in I^{2nN}$.   Let $i$ be any such index, and let $m\geq 1$ be large enough so that $a_i\in A_m$, and that (one may need to make $m$ larger still) $a_i^2\in I^{2nN}$ holds true when one regards $I$ as the ideal of $A_m$ generated by $X_1^{(1)},\dots,X_n^{(1)}$.  Let $I_m$ be the maximal ideal of $A_m$.  By Lemma \ref{idealinclusions}, we have $a_i^2\in I_m^{2nq^{(m-1)n}N}$.  Since $A_m$ is a regular local ring (in fact it is a power series ring over $k$), we have $a_i\in I_m^{nq^{(m-1)n}N}$.  Using Lemma \ref{idealinclusions} once again, we have $a_i\in I^N$.  This shows that $a=\lim_{i\to\infty} a_i= 0$ in $A$.  
\end{proof}

We continue with a description of $A/I$.
\begin{prop}
\label{AmodI}
We have an isomorphism of $k$-algebras
\[ A/I \isom k[X_1^{1/q^\infty},\dots,X_n^{1/q^\infty}]/(X_1,\dots,X_n). \]
\end{prop}
\begin{proof}
We have $A/I=\varinjlim A_m/I$.  Let $R$ be a $k$-algebra.  By Prop. \ref{Iext}, $\Hom(A/I,R)=\varprojlim \Hom(A_m/I,R)$ is identified with the set of $n$-tuples of elements of $\varprojlim G(R)$ whose projection onto the first coordinate is zero.  Under the isomorphism in Eq.~\ref{Glimit}, the set of such elements of $\varprojlim G(R)$ is carried onto the set of elements of $\varprojlim_{\tau}\Nil(R)$ whose projection onto the first coordinate is zero.  We now see that $\Spec A/I$ represents the same functor as $\Spec k[X_1^{1/q^\infty},\dots,X_n^{1/q^\infty}]/(X_1,\dots,X_n)$.
\end{proof}

\subsection{The case of height 1}

Suppose that $G_0$ is a formal $\OO_K$-module of height 1 over $k$;  then classical Lubin-Tate theory shows that there is a unique deformation of $G_0$ to any object in $\mathcal{C}$.  Thus $A_0=\OO_L$.  Let $G$ be the deformation of $G_0$ to $\OO_L=k\powerseries{\pi}$.

For an object $R$ of $\mathcal{C}$, a Drinfeld basis for $G[\pi](R)$ is an element $\lambda$ of the maximal ideal of $R$ such that $T^q-\lambda^{q-1}T$ is divisible by $[\pi]_G(T)$ in $R\powerseries{T}$.  Write $[\pi]_G(T)=T\Phi(T)$;  then $T^{q-1}-\lambda^{q-1}$ is divisible by $\Phi(T)$, and therefore $\lambda$ is a root of $\Phi(T)$.  By the Eisenstein criterion, $\Phi(T)$ equals a unit times an irreducible polynomial in $\OO_L[T]$ of degree $q-1$.    Similarly, for $m\geq 1$, if $\lambda_m$ is a Drinfeld basis for $G[\pi^m](R)$, then $\lambda_m$ is a root of $\Phi_m(T)=\Phi([\pi^{m-1}]_G(T))$, which is a unit times an irreducible polynomial of degree $(q-1)q^{m-1}$.

Thus $A_m=\OO_L\powerseries{T}/\Phi_m(T)$ equals the ring of integers in the extension field $L_m/L$ obtained by adjoining the $\pi^m$-torsion in $G$ to $L$.  Let $L_\infty$ be the union of the valued fields $L_m$, and let $\hat{L}_\infty$ be its completion.  We have $\mathcal{M}_{G_0}=\Spf \OO_{\hat{L}_\infty}$.

For an object $R$ of $\Nilp_{\OO_L}$, a compatible system of Drinfeld bases in $\varprojlim_m G[\pi^m](R)$ gives an element of $\tilde{G}(R)=\varprojlim G(R)$.
Thus there is a natural morphism $\mathcal{M}_{G_0}\to\tilde{G}$ of formal schemes over $\Spf \OO_L$.  This corresponds to a continuous homomorphism of $\OO_L$-algebras $\OO_L\powerseries{t^{1/q^\infty}}\to \OO_{\hat{L}_\infty}$.  Let $\varpi$ be the image of $t$ in $\OO_{\hat{L}_\infty}$.  We get a map $\phi\from k\powerseries{t^{1/q^\infty}}\to \OO_{\hat{L}_\infty}$ given by $t\mapsto \varpi$.


\begin{prop} The map $\phi$ is an isomorphism.  Thus $\hat{L}_\infty$ is a perfect field, isomorphic to $k\laurentseries{t^{1/q^\infty}}$, the $t$-adic completion of the perfect closure of $k(t)$.
\end{prop}

\begin{proof} 
Since $\Phi_m(T)$ modulo $\pi\OO_L\powerseries{T}$ is a unit times $T^{q^{m-1}(q-1)}$, the same is true modulo $\lambda A_m\powerseries{T}$, and we have an isomorphism $A_m/\lambda \isom k\powerseries{t}/t^{q^{m-1}}$.  
Thus the homomorphism $\OO_L\powerseries{t^{1/q^\infty}}\to \OO_{\hat{L}_\infty}$ induces an isomorphism $k[t^{1/q^\infty}]/t\to \OO_{\hat{L}_\infty}/\varpi$. 

This shows that $\phi$ is surjective:  any $b\in \OO_{\hat{L}_\infty}$ can be written $b=\phi(a_1)+\varpi b_1=\phi(a_1)+\phi(t)b_1$.  But then we can write $b_1=\phi(a_2)+\phi(t)b_2$, and so forth, the result being that $b=\phi(a_1+ta_2+\dots)$. 

For injectivity, suppose $\phi(a)=0$.  Since $k\powerseries{t^{1/q^\infty}}$ is perfect and $\OO_{\hat{L}_\infty}$ is reduced, we have $\phi(a^{1/q^m})=0$ for all $m\geq 1$.  Since $\phi$ modulo $t$ is an isomorphism, we have that $a^{1/q^m}$ is divisible by $t$ for all $m$, which shows that $a=0$.   
\end{proof}


\subsection{Determinants of truncated BT modules}

Returning to the general case, $G_0$ is once again a formal $\OO_K$-module of height $n$.

\begin{prop}
\label{mu_m}
\begin{enumerate}
\item There exists a formal $\OO_K$-module $\wedge G_0$ of dimension 1 and height 1, such that for all $m\geq 1$, $\wedge G_0[\pi^m]$ is the top exterior power of $G_0[\pi^m]$.  That is, the category of $\OO_K$-multilinear alternating maps of $\OO_K$-module schemes out of $G_0[\pi^m]^n$ admits an initial object
\[ \mu_{m,0}\from G_0[\pi^m]^n\to \wedge G_0[\pi^m]. \]
\item  Assume that $k=\overline{k}$.  Let $R\in\mathcal{C}$, and let $G$ be a deformation of $G_0$ to $R$.  Let $\wedge G$ be the unique deformation of $\wedge G_0$ to $R$.  The family $\mu_{m,0}$ lifts to a compatible family of $\OO_K$-alternating maps of $\OO_K$-module schemes
\[ \mu_m\from G[\pi^m]^n\to \wedge G[\pi^m].\]
Further, $\mu_m$ has the property that if $x_1,\dots,x_n$ is a Drinfeld basis for $G[\pi^m](R)$, then $\mu_m(x_1,\dots,x_n)$ is a Drinfeld basis for $\wedge G[\pi^m](R)$.
\end{enumerate}
\end{prop}

Part 1 is a special case of the main theorem of the thesis of Hadi Hedayatzadeh \cite{Hedayatzadeh}, who constructs arbitrary exterior powers of $\pi$-divisible modules of dimension $\leq 1$ over a field.  The methods there could probably prove Part 2 as well even without the hypothesis that $k$ be algebraically closed.  We prove both parts in \S\ref{determinantbyhand}.  





Letting $m\to\infty$ in Prop.~\ref{mu_m}, we obtain an $\OO_K$-alternating map
\[ \mu\from \varprojlim G[\pi^m]^n \to \varprojlim \wedge G[\pi^m] \]
between functors $\Nilp_{R}\to \Mod_{\OO_K}$.  Tensoring with $K$ and applying Lemma \ref{Tatemod}, we get a $K$-alternating map $\tilde{G}^n\to \widetilde{\wedge G}$ of functors $\Nilp_{R}\to \Vect_K$, which we also call $\mu$.
Now by part (2) of Lemma~\ref{reductionmap}, $\tilde{G}$ and $\widetilde{\wedge G}$ are formal $K$-vector spaces over $R$.  We also have the top exterior power $\wedge\tilde{G}$ of the formal $K$-vector space $\tilde{G}$, as in \S\ref{multilinear}.  By the universality of $\wedge\tilde{G}$, we may deduce:
\begin{prop}  \label{detsagree} There is an isomorphism $\wedge \tilde{G}\to \widetilde{\wedge G}$ of formal $K$-vector spaces over $R$ making the diagram commute:
\[
\xymatrix{
& \wedge \tilde{G} \ar[dd] \\
\tilde{G}^n \ar[ur]^{\delta} \ar[dr]_{\mu} & \\
& \widetilde{\wedge G}
}
\]
\end{prop}

\subsection{The main theorem}
Recall that $G^{\univ}$ is the universal deformation of $G_0$ to $A_0$.  We assume henceforth that $k$ is algebraically closed.  Lemma \ref{mu_m} gives us the top exterior power $\wedge G^{\univ}$.  All the same we could have considered a lift $G$ of $G_0$ to $\OO_L$, and taken its top exterior power $\wedge G$.  Since both $\wedge G^{\univ}$ and $\wedge G$ are deformations of the same height 1 formal $\OO_K$-module $\wedge G_0$, we must have $\wedge G^{\univ}=\wedge G\otimes_{\OO_L} A_0$.

We have a universal Drinfeld basis for $G^{\univ}[\pi^m](A_m)$.  Passing to the limit as $m\to\infty$, we get elements $x_1,\dots,x_n\in \varprojlim_\pi G^{\univ}(A) = \tilde{G}^{\univ}(A)$.  The determinant $\delta(x_1,\dots,x_n)$ constitutes a Drinfeld basis for $\wedge G$ over $A$ (Lemma~\ref{mu_m}).  This in turn determines a morphism $\mathcal{M}_{G_0}\to \mathcal{M}_{\wedge G_0}$.  By Prop.~\ref{detsagree}, we have the diagram of formal schemes over $\Spf \OO_L$:

\begin{equation}
\label{Mdiagram}
\xymatrix{
\mathcal{M}_{G_0} \ar[r] \ar[d] & \mathcal{M}_{\wedge G_0} \ar[d] \\
\tilde{G}^n \ar[r]_{\delta} & \wedge \tilde{G}
}
\end{equation}
The main theorem of the paper is:
\begin{Theorem}\label{mainthm}
The above diagram is Cartesian.  That is, $\mathcal{M}_{G_0}$ is isomorphic to the fiber product of $\tilde{G}^n$ and $\mathcal{M}_{\wedge G_0}$ over $\wedge \tilde{G}$.
\end{Theorem}
The proof will occupy us for the rest of the section.  Here is an overview.  The fiber product of $\tilde{G}^n$ and $\mathcal{M}_{\wedge G_0}$ over $\wedge\tilde{G}$ is an affine formal scheme, say $\Spf B$.  We have a map $\phi\from B\to A$ which we claim is an isomorphism.

The proof will proceed by establishing the following facts:
\begin{enumerate}
\item $A$ is perfect.
\item The Frobenius map is surjective on $B$.
\item There exists a finitely generated ideal of definition $J$ of $B$ such that $\phi$ descends to an isomorphism $B/J\to A/I$.
\end{enumerate}

After these facts are established, we can show that $B\to A$ is an isomorphism.  For each $m\geq 1$, the $q^m$th power Frobenius map gives an isomorphism $\tau^m\from A/I\to A/I^{[q^m]}$.   The same Frobenius map gives a surjection $B/J\to B/J^{[q^m]}$.  The diagram
\[
\xymatrix{B/J \ar[r]^{\sim} \ar[d]_{\tau^m} & A/I \ar[d]^{\tau^m} \\
B/J^{[q^m]} \ar[r] & A/I^{[q^m]} }
\]
now shows that $B/J^{[q^m]}\to A/I^{[q^m]}$ is an isomorphism.  The rings $A$ and $B$ are complete with respect to the $I$-adic and $J$-adic topologies, respectively, and because these ideals are finitely generated, the sequences $I^{[q^m]}$ and $J^{[q^m]}$ generate the same topologies.  Thus $B\to A$ is an isomorphism.

\subsection{$A$ is perfect}

Consider the power series $[\pi]_{G_0}(T)$ as lying in $A\powerseries{T}$ via the inclusion $k\injects A$. We can do slightly better than Eq.\eqref{piGmodI}:

\begin{lemma} \label{modI2} $[\pi]_{G^{\univ}}(T)\equiv [\pi]_{G_0}(T)\mod I^2\powerseries{T}$.
\end{lemma}

\begin{proof} Let $I_0$ be the maximal ideal of $A_0$.  The congruence holds modulo $I_0\powerseries{T}$, so this is a matter of showing that $I_0\subset I^2$.  The quotient $A_1/I_0$ classifies Drinfeld bases for $G_0[\pi]$.  For a $k$-algebra $R$, an $n$-tuple $x_1,\dots,x_n\in R$ is a Drinfeld basis for $G_0[\pi](R)$ when the product
\[ \prod_{(a_1,\dots,a_n)}\left(T-\sum_i a_i x_i\right)=T^{q^n}+u_{n-1}T^{q^{n-1}}+\dots+u_0 \]
is divisible by $[\pi]_{G_0}(T)$ in $R\powerseries{T}$, which is true if and only if each of the coefficients $u_i=0$ for $i=0,\dots,n-1$.  Here we have $u_i=U_i(x_1,\dots,x_n)$ for a homogeneous polynomial $U_i\in k\powerseries{X_1,\dots,X_n}$ of degree at least 2.  We find that
\[ A_1/I_0 = k[X_1,\dots,X_n]/(U_1,\dots,U_{n-1}) \]
where $X_i=X_i^{(1)}$,  Since $I=(X_1,\dots,X_n)$, we find that $I_0\subset I^2$.
\end{proof}

Let
\[ Y_i = \lim_{m\to\infty} [\pi^m]_{G_0}\left(X_i^{(m-1)}\right) \in A. \]
By Lemma~\ref{modI2} we have $Y_i\in I$.  Note that $[\pi^m]_{G_0}(T)$ is a power series in $k\powerseries{T^{q^{mn}}}$.  It follows that $Y_i$ admits a system of $q^m$th roots $Y_i^{1/q^m}\in A$.  Let $\mathcal{I}\subset I$ be the ideal generated by the $Y_i$.
\begin{lemma} \label{II} $\mathcal{I}=I$.
\end{lemma}

\begin{proof} For $i=1,\dots,n$ we have $Y_i-[\pi^m]_{G_0}\left(X_i^{(m-1)}\right)\in I^2$ for all sufficiently large $m$.  By Lemma \ref{modI2},
$X_i^{(1)}=[\pi^m]_{G^{\univ}}\left(X_i^{(m-1)}\right)\equiv [\pi]_{G_0}\left(X_i^{(m-1)}\right)$ modulo $I^2$.  Thus $X_i^{(1)}- Y_i\in I^2$, and therefore $X_i^{(1)}\in \mathcal{I}+I^2$. Since the $X_i^{(1)}$ generate $I$, we have $I\subset \mathcal{I}+I^2$.

Since $I$ is a finitely generated ideal, $I/\mathcal{I}$ is a finitely generated $A$-module.  The inclusion $I\subset\mathcal{I}+I^2$ shows that $I\cdot (I/\mathcal{I})=I/\mathcal{I}$.  By Nakayama's lemma, $I/\mathcal{I}$ is annihilated by some element of $1+I$.  But $1+I\subset A^\times$, so that $I/\mathcal{I}=0$ and therefore $I=\mathcal{I}$.  
\end{proof}

\begin{prop}\label{Aperfect} $A$ is perfect.
\end{prop}

\begin{proof} $A$ is reduced (Prop. \ref{Areduced}), so it suffices to show that Frobenius is surjective on $A$.  From Prop. \ref{AmodI} we see that Frobenius is surjective on $A/I$.  From Lemma~\ref{II} we see that $I$ is generated by finitely many elements $Y_i$ which admit arbitrary $q$th power roots.  These facts suffice.
\end{proof}

Applying Prop. \ref{Aperfect} in the case of $n=1$ shows that the field $\hat{L}_\infty$ is perfect.

\subsection{The Frobenius map is surjective on $B$}
By Lemma~\ref{FormalVSadic}, we have isomorphisms of formal schemes
\[\tilde{G}^n\approx \Spf \OO_L\powerseries{X_1^{1/q^\infty},\dots,X_n^{1/q^\infty}}\]
and
\[\wedge\tilde{G}^n\approx \Spf \OO_L\powerseries{X^{1/q^\infty}}.\]
Therefore we have the following presentation of $B$:
\[ B=\OO_L\powerseries{X_1^{1/q^\infty},\dots,X_n^{1/q^\infty}}\hat{\otimes}_{\OO_L\powerseries{X^{1/q^\infty}}} \OO_{\hat{L}_\infty}. \]
Since $\OO_{\hat{L}_\infty}$ is a perfect $\OO_L$-algebra, it is clear from this presentation that the $q$th power Frobenius map is surjective on $B$.

\subsection{The isomorphism $B/J\to A/I$}

Let $G$ be an honest lift of $G_0$ to $\OO_L$, so that we may identify $\tilde{G}$ with the universal cover of $G$(cf. the discussion at the end of \S\ref{univcover}). Thus for $R\in\Nilp_{\OO_L}$ we have $\tilde{G}(R)=\varprojlim_\pi G(R)$.  For an element $x\in\tilde{G}(R)$, we write $(x^{(1)},x^{(2)},\dots)$ for the corresponding sequence of elements of $G(R)$.  Similarly, we have $\wedge G$, the unique lift of $\wedge G_0$;  we use a similar convention for elements of $\wedge \tilde{G}(R)=\widetilde{\wedge G}(R)$.

Recall that $\Spf B$ was defined as the fiber product of $\tilde{G}^n$ and $\mathcal{M}_{\wedge G_0}$ over $\wedge\tilde{G}$.  Thus for an object $R$ of $\Nilp_{\OO_L}$, $\Hom(B,R)$ is in bijection with the set of $n$-tuples $x_1,\dots,x_n\in\tilde{G}^n$ such that $\delta(x_1,\dots,x_n)$, a priori just an element of $\wedge\tilde{G}(R)$, actually lies in $\varprojlim \wedge G[\pi^m](R)$, with each $\delta(x_1,\dots,x_n)^{(m)}$ constituting a Drinfeld basis for $\wedge G[\pi^m](R)$.  This last condition is equivalent to the condition that $\delta(x_1,\dots,x_n)^{(1)}$ be a Drinfeld basis for $\wedge G[\pi](R)$.

Recall that $I\subset A$ is the ideal $(X_1^{(1)},\dots,X_n^{(1)})$.  This ideal contains $\pi$.  In the proof of Prop. \ref{AmodI}, we gave the following moduli interpretation for the closed subscheme $\Spec A/I\subset \Spf A$:  For a $k$-algebra $R$, $\Hom(A/I,R)$ is in bijection with the set of $n$-tuples of elements $x_1,\dots,x_n\in\varprojlim G_0(R)$ with $x_i^{(1)}=0$, $i=1,\dots,n$.

Let $J\subset B$ be the finitely generated ideal corresponding to the conditions $\pi=0$ and $x_i^{(1)}=0$, for $i=1,\dots,n$.  That is, for a $k$-algebra $R$, $\Hom(B/J,R)$ is in bijection with the set of $n$-tuples $x_1,\dots,x_n\in\tilde{G}^n$ such that $x_i^{(1)}=0$, with the additional condition that $\delta(x_1,\dots,x_n)^{(1)}$ be a Drinfeld basis for $\wedge G[\pi](R)$.

By Prop. \ref{detsagree}, we have
\[ \delta(x_1,\dots,x_n)^{(1)}=\mu_1(x_1^{(1)},\dots,x_n^{(1)})=0.\]
However, we have already seen (cf. the proof of Prop.~\ref{Iext}) that 0 constitutes a Drinfeld basis for $\wedge G[\pi](R)$ for any $k$-algebra $R$.  Thus the ``additional condition" of the previous paragraph is automatically satisfied.  From the functorial descriptions of $\Spec B/J$ and $\Spec A/I$ above, we see that $B/J\to A/I$ is an isomorphism.

\section{Some explicit calculations}

\subsection{The standard formal $F$-vector space}

\begin{defn}  Let $n\geq 1$.  The {\em standard formal $K$-vector space} of height $n$ over $k$ is the one-dimensional formal $K$-vector space $\V$ which arises from the continuous $\FF_q$-algebra homomorphism $K\to k\laurentseries{\pi}$ which sends $\pi$ to $\tau^n$.  Thus if $R$ is a $k$-algebra, $\V(R)$ is equal to the $K$-vector space $\varprojlim_\tau\Nil(R)$, where $\pi$ acts through $X\mapsto X^{q^n}$.
\end{defn}

Let $\V$ be the standard formal $K$-vector space of height $n$ over $\FF_q$.  Let $e\in M(V)$ represent the identity morphism in $\Hom_{\FF_q}(\V,\bG_{\FF_q}^{\flat})$.  Then $M(\V)$ has $K$-basis $\set{F^i(e)}$, $i=0,\dots,n-1$.  With respect to this basis, the matrix of $F$ is
\[
\begin{pmatrix}
0 & 1 & & & \\
& 0 & 1 & &\\
&   & \ddots  & & \\
& & & 0 & 1\\
\pi & & & &0
\end{pmatrix}
\]

Thus the top exterior power $\wedge M(\V)$ is an isocrystal of dimension 1, where $F$ acts as the scalar $(-1)^{n-1}\pi$.  Thus $\pi$ acts on $\wedge \V(R)=\varprojlim_\tau\Nil(R)$ by $X\mapsto (-1)^{n-1}X^q$.

Write $\varpi$ for the $q$th power Frobenius endomorphism of $\V$, so that $\varpi^n=\pi$.  Then the endomorphism algebra $\End \V\otimes\mathbf{F}_{q^n}$ is the noncommutative ring of power series $\sum_{m\gg -\infty} a_m\varpi^m$, with $a_m\in\FF_{q^n}$, subject to the rule $\varpi a=a^q\varpi$, for $a\in\FF_{q^n}$.  Call this algebra $D$:  it is the central division algebra over $K$ of invariant $1/n$.  Write $N\from D\to K$ for the reduced norm.

It will be useful to have an expression for the determinant map $\delta\from \V^n\to \wedge\V$.  For indeterminates $T_1,\dots,T_n$ in any $\FF_q$-algebra, let $\Delta(T_1,\dots,T_n)$ be the Moore determinant:
\[ \Delta(T_1,\dots,T_n) = \det\left(X_i^{q^j}\right). \]
Then $\Delta$ induces an $\FF_q$-alternating (but certainly not $K$-alternating) map $\V^n\to\wedge \V$ which we also call $\Delta$.
\begin{lemma}  The determinant map $\delta\from \V^n\to\wedge \V$ is given by
\[ \delta(x_1,\dots,x_n)=\sum_{(a_1,\dots,a_n)} \Delta\left(\pi^{a_1}x_1,\dots,\pi^{a_n}x_n\right). \]
Here the sum ranges over $n$-tuples of integers $(a_1,\dots,a_n)$ satisfying $\sum_i a_i= 0$.
\end{lemma}

Finally, note that $\tilde{G}^n$ has an action of the product $M_n(K)\times D$.  For a given element $(g,b)\in M_n(K)\times D$, let $x=\det(g)N(b)$.  Considering the action on the coordinate module of $\tilde{G}^n$, we see that the diagram
\begin{equation}
\label{gbcommutes}
\xymatrix{
\tilde{G}^n \ar[d]_{\delta} \ar[r]^{(g,b)} & \tilde{G}^n \ar[d]^{\delta} \\
\wedge\tilde{G} \ar[r]_{x} & \wedge\tilde{G}
}
\end{equation}
commutes.

\subsection{The standard formal $\OO_K$-module, and its universal deformation}
\label{determinantbyhand}

Let $G_0$ be the formal $\OO_K$-module over $\overline{\FF}_q$ whose underlying $\FF_q$-module is $\hat{\mathbf{G}}_a$ and for which multiplication by $\pi$ is $[\pi]_{G_0}(X)=X^{q^n}$.  Then $G_0$ has height $n$.  Then $\OO_D=\End G_0$ is the ring of integers in a division algebra $D/K$ of invariant $1/n$.

We have $A_0=\overline{\FF}_q\powerseries{u_1,\dots,u_{n-1}}$.  The universal deformation $G^{\univ}$ of $G_0$ is the one for which
\[ [\pi]_{G^{\univ}}(X) = \pi X  + u_1X^q + \dots + u_{n-1}X^{q^{n-1}} + X^{q^n}. \]

In this situation we can prove Prop. \ref{mu_m} ``by hand".  It suffices to construct the compatible family of alternating forms $\mu_m$ on the universal deformation $G^{\univ}$.  Let $\wedge G$ be the one-dimensional formal $\OO_K$-module over $\OO_K$ defined by
\[ [\pi]_{\wedge G}(X)=\pi X + (-1)^{n-1} X^q. \]
The Moore determinant determines an $\FF_q$-alternating map $G^n\to \wedge G$ which we also call $\Delta$.
Then the desired form $\mu_m$ is
\[ \mu_m(X_1,\dots,X_n)=\sum_{(a_1,\dots,a_n)} M\left(\pi^{a_1}X_1,\dots,\pi^{a_n}X_n\right), \]
where the sum runs over tuples of integers $(a_1,\dots,a_n)$ with $0\leq a_i< m$ and $\sum_i a_i=(m-1)(n-1)$.
The claims about $\mu_m$ appearing in Prop.~\ref{mu_m} are established in \cite{Wei}, Lemma 3.4 and Prop. 3.7.

For each $m\geq 1$, let $L_m$ be the splitting field of $[\pi^m]_{\wedge G}(X)$ over $L=\hat{K}^{\nr}=\overline{\FF}_q\laurentseries{\pi}$.  Let $t_m\in \OO_{L_m}$ be a primitive root of $[\pi^m]_{\wedge G}(X)$.
Choose these so that $[\pi]_{\wedge G}(t_m)=t_{m-1}$.  
Let $\hat{L}_\infty$ be the completion of $L_\infty =\bigcup_{m\geq 1} L_m$.  By Lubin-Tate theory, $\hat{L}_\infty=\hat{K}^{\ab}$, the completion of the maximal abelian extension of $K$. Then the limit $\lim_{m\to \infty} (-1)^{(n-1)m} t_m^{q^m}$ converges to an element $t\in \OO_{\hat{L}_\infty}$, all of whose $q$th power roots lie in $\OO_{\hat{L}_\infty}$.
Then $\OO_{\hat{L}_\infty}=\overline{\FF}_q\powerseries{t^{1/q^\infty}}$.

Thm.~\ref{mainthm} translates into the following description of $A$.
\begin{Theorem}  We have an isomorphism of topological $\overline{\FF}_q$-algebras
\[ \overline{\FF}_q\powerseries{X_1^{1/q^\infty},\dots,X_n^{1/q^\infty}}\to A \]
which sends $X_i$ to $\lim_{i\to\infty} \left(X_i^{(m)}\right)^{q^m}$.
\end{Theorem}

\begin{proof} By Thm.~\ref{mainthm},  we have an isomorphism of complete local $\OO_L$-algebras
\[ \OO_L\powerseries{X_1^{1/q^\infty},\dots,X_n^{1/q^\infty}}\hat{\otimes}_{\OO_L\powerseries{X^{1/q^\infty}}} \OO_{\hat{L}_\infty}\to A,\]
where
\begin{itemize}
\item The map $\OO_L\powerseries{X^{1/q^\infty}}\to \OO_L\powerseries{X_1^{1/q^\infty},\dots,X_n^{1/q^\infty}}$ sends $X$ to
\begin{equation}
\label{deltadef}
\delta(X_1,\dots,X_n)=\sum_{a_1+\dots+a_n=0} \Delta\left(X_1^{q^{a_1n}},\dots,X_n^{q^{a_nn}}\right), 
\end{equation}
\item The map $\OO_L\powerseries{X^{1/q^\infty}}\to \OO_{\hat{L}_\infty}$ sends $X$ to $t$,
\item The map into $A$ sends $X_i\otimes 1$ to $\lim_{m\to\infty} \left(X_i^{(m)}\right)^{q^m}$.
\end{itemize}
The map $\OO_L\powerseries{X^{1/q^\infty}}\to \OO_{\hat{L}_\infty}$ is surjective with kernel generated by $\pi-g(X)$ for some fractional power series $g(X)\in \overline{\FF}_q\powerseries{X^{1/q^\infty}}$ without constant term.   Thus
\begin{eqnarray*}
A &\isom& \OO_L\powerseries{X_1^{1/q^\infty},\dots,X_n^{1/q^\infty}}/\left(\pi - g(\delta(X_1,\dots,X_n)\right) \\
&=& \overline{\FF}_q\powerseries{X_1^{1/q^\infty},\dots,X_n^{1/q^\infty}}.
\end{eqnarray*}
\end{proof}

\subsection{The group actions}
\label{groupactions}
Let us now assume that $G_0$ is the (unique) formal $\OO_K$-module of dimension 1 and height $n$ over $\overline{\FF}_q$.   As before, let $A_m$ be the local ring governing deformations of $G_0$ with Drinfeld $\pi^m$ level structures, and let $A$ be the completion of $\varinjlim A_m$ with respect to the maximal ideal of $A_1$.  Then $A$ is a complete local $\OO_{\hat{K}^{\nr}}$-algebra, where $\hat{K}^{\nr}=\overline{\FF}_q\laurentseries{\pi}$.  Let $\mathcal{M}=\Spf A$.  Thm.~\ref{mainthm} shows that $\mathcal{M}$ is the fiber product
\[
\xymatrix{
\mathcal{M} \ar[r] \ar[d] & \mathcal{M}_{\wedge G_0} \ar[d] \\
\tilde{G}^n \ar[r]_{\delta} & \wedge \tilde{G}
}
\]
Let $\OO_D=\End G_0$, so that $D=\OO_D[1/\pi]$ is the division algebra of invariant $1/n$.   Then $\tilde{G}^n$, being a formal $K$-vector space, admits a left action of $\GL_n(K)\times D^\times$.  Since we are ultimately interested in the actions of groups on cohomology, where we would prefer left actions, we use the transpose on $\GL_n(K)$ and the inversion map on $D^\times$ to get a right action of $\GL_n(K)\times D^\times$ on $\tilde{G}^n$.  Let $(\GL_n(K)\times D^\times)^{(0)}$ be the subgroup $\GL_n(K)\times D^\times$ consisting of pairs $(g,b)$ for which $\det(g)N(b)^{-1}\in \OO_K^\times$.  By Eq. \eqref{gbcommutes}, we get an action of $(\GL_n(K)\times D^\times)^{(0)}$ on $\mathcal{M}$, which is fibered over the action of $\OO_K^\times$ on $\mathcal{M}_{\wedge G_0}=\Spf \OO_{\hat{K}^{\ab}}$ via the map $(g,b)\mapsto \det(g)N(b)^{-1}$.

Let $\C$ be the completion of an algebraic closure of $\hat{K}^{\nr}$, and consider $\mathcal{M}_{\OO_{\C}}=\Spf A\hat{\otimes}_{\OO_{\hat{K}^{\nr}}} \OO_{\C}$.  Then of course $\mathcal{M}_{\OO_{\C}}$ admits a right action of $(\GL_n(K)\times D^\times)^{(0)}\times I_K$, where $I_K=\Gal(\overline{K}/K^{\nr})$ is the inertia group of $K$.  But in fact $\mathcal{M}_{\OO_{\C}}$ admits an action of the larger group $(\GL_n(K)\times D^\times \times W_K)^{(0)}$, which we define as the kernel of the map
\begin{eqnarray*}
\GL_n(K)\times D^\times\times W_K&\to& \Z \\
(g,b,w) &\mapsto& v\left(\det(g)N(b)^{-1}\chi(w)^{-1}\right),
\end{eqnarray*}
where $\chi\from W_K^{\ab}\to K^\times$ is the inverse of the isomorphism of local class field theory and $v$ is the valuation on $K$.  This action is defined as follows.  If $w\in W_K$ is an element lying over the $q^m$th power Frobenius map, then $w$ induces a morphism of formal schemes $\tilde{G}_{\OO_C}\to \tilde{G}^{(q^m)}_{\OO_C}$ which lies over $w\from \Spf \OO_{\C}\to\Spf\OO_{\C}$.   On the other hand, the $q^m$th power Frobenius isogeny $G_0\to G_0^{(q^m)}$ induces an isomorphism $\tilde{G}_{\OO_C}\to\tilde{G}^{(q^m)}_{\OO_C}$ of formal schemes over $\Spf \OO_C$.  Composing the first map with the inverse of the second gives a map $\tilde{G}_{\OO_C}\to\tilde{G}_{\OO_C}$.   Explicitly, the action of $w$ on $\tilde{G}_{\OO_C}=\Spf\OO_C\powerseries{X^{1/q^\infty}}$ is as follows:
it acts on the scalar ring $\OO_C$ via its natural action, and it sends $X$ to $X^{q^{-m}}$.  Then $\tilde{G}_{\OO_C}^n$ gets an action of the triple product group $\GL_n(K)\times D^\times\times W_K$;  a triple $(g,b,w)$ preserves the formal subscheme $\mathcal{M}_{\OO_{\C}}$ exactly when $\det(g)N(b)^{-1}\chi(w)^{-1}\in \OO_K^\times$.

The existence of the map $\mathcal{M}\to\Spf\OO_{\hat{K}^{\ab}}$ shows that $\mathcal{M}_{\OO_{\C}}$ is fibered over $\Spf\OO_{\hat{K}^{\ab}}\hat{\otimes}_{\OO_{\hat{K}^{\nr}}}\OO_{\C}$.
Let $\mathcal{E}$ be the set of continuous $\hat{K}^{\nr}$-linear field embeddings $\hat{K}^{\ab}\to \C$.  Then $\mathcal{E}$ is a principal homogeneous space for $\Gal(K^{\ab}/K^{\nr})\isom \OO_K^\times$, and as such it has a natural topology.  Let $C(\mathcal{E},\OO_{\C})$ be the $\OO_{\C}$-algebra of continuous maps from $\mathcal{E}$ into $\OO_{\C}$. We have an isomorphism of $\OO_C$-algebras
\begin{eqnarray*}
\OO_{\hat{K}^{\ab}}\hat{\otimes}_{\OO_{\hat{K}^{\nr}}} \OO_{\C} &\to& C(\mathcal{E},\OO_C) \\
\alpha\otimes \beta &\mapsto& \left(\sigma\mapsto \beta\sigma(\alpha)\right) 
\end{eqnarray*}

Choose a $\hat{K}^{\nr}$-linear embedding $\hat{K}^{\ab}\to \C$ and let $A^\circ=A\hat{\otimes}_{\OO_{\hat{K}^{\ab}}}\OO_{\C}$.  Then $A\hat{\otimes}_{\OO_{\hat{K}^{\nr}}}\OO_{\C}$ may be identified with $C(\mathcal{E},A^\circ)$.
Note that $A^\circ$ has an action of $(\GL_n(K)\times D^\times \times W_K)^{\circ}$, the group of triples $(g,b,w)$ with $\det(g)N(b)^{-1}\chi(w)^{-1}=1$.

From the description of $A$ in Thm.~\ref{mainthm}, we have an isomorphism
\begin{equation}
\label{Acirc}
A^\circ \isom \OO_\C\powerseries{X_1^{1/q^\infty},\dots,X_n^{1/q^\infty}}/\left(\delta(X_1,\dots,X_n)^{1/q^m}-t^{1/q^m}\right)_{m\geq 1},
\end{equation}
where $\delta(X_1,\dots,X_n)$ is the element of Eq.~\eqref{deltadef}.  

\section{The Lubin-Tate perfectoid space, and special affinoids}
\label{affinoids}

Let $\mathcal{M}^{\ad}$ be the adic space associated to the formal scheme $\mathcal{M}$.  That is, $\mathcal{M}^{ad}$ is the set of continuous valuations on $A\approx \overline{\FF}_q\powerseries{X_1^{1/q^\infty},\dots,X_n^{1/q^\infty}}$.  The existence of the map $\OO_{\hat{K}^{\ab}}\to A$ shows that $\mathcal{M}^{\ad}$ is fibered over $\Spa(\OO_{\hat{K}^{\ab}},\OO_{\hat{K}^{\ab}})$.  Let $\eta=\Spa(\hat{K}^{\ab},\OO_{\hat{K}^{\ab}})$ be the generic point of $\Spa(\OO_{\hat{K}^{\ab}},\OO_{\hat{K}^{\ab}})$, and let $\mathcal{M}^{\ad}_\eta$ be the generic fiber.  Thus $\mathcal{M}^{\ad}_{\eta}$ is the set of continuous valuations $\abs{\;}$ on $A$ for which $\abs{\pi}\neq 0$.

The field $\hat{K}^{\ab}$ is a complete and perfect nonarchimedean field whose valuation is rank 1 and non-discrete.  It is thus an example of a {\em perfectoid field}, cf. \cite{Sch}, Defn. 3.1.

\begin{lemma} $\mathcal{M}^{\ad}_{\eta}$ is covered by affinoids of the form $\Spa(R,R^+)$, where $R$ is a perfect Banach $\hat{K}^{\ab}$-algebra whose set of power-bounded elements is open and bounded.
\end{lemma}

\begin{proof} This is a consequence of the fact that $A$ is a perfect complete flat $\OO_{\hat{K}^{\ab}}$-algebra admitting a finitely generated ideal of definition $(X_1,\dots,X_n)$.  A family of affinoids which cover $\mathcal{M}^{\ad}$ is given by the rational subsets defined by $\abs{X_i}^N\leq \abs{\pi}$, for $N\geq 0$.  One takes $R=A\tatealgebra{X_1^N/\pi,\dots,X_n^N/\pi}[1/\pi]$, and $R^+$ the integral closure of $A\tatealgebra{X_1^N/\pi,\dots,X_n^N/\pi}$ in $R$.  The Banach norm on $R$ is the one induced by $R^+$:  $\norm{x}=
\inf\set{\abs{t},\; t\in\hat{K}^{\ab},\;tx\in R^+}$.  Under this norm, $R^+$ is open and bounded, and the set of power-bounded elements of $R$ is $R^+$.
\end{proof}

By Prop. 5.9 of \cite{Sch}, the rings $R$ appearing in the lemma are {\em perfectoid $K$-algebras} (Defn. 5.1), so that each $(R,R^+)$ is a {\em perfectoid affinoid $K$-algebra} (Defn. 6.1).  The adic spectrum $\Spa(R,R^+)$ is an affinoid perfectoid space, and therefore $\mathcal{M}^{\ad}_{\eta}$ is a {\em perfectoid space} (Defn. 6.16), which we call the {\em Lubin-Tate perfectoid space}.

Recall that $\C$ is the completion of an algebraic closure of $\hat{K}^{\nr}$.  We write $\overline{\eta}=\Spa(\C,\OO_{\C})$ and $\mathcal{M}^{\ad}_{\overline{\eta}}$ for the geometric adic generic fiber of $\mathcal{M}$.

Fiber products exist in the category of perfectoid spaces (\cite{Sch}, Prop. 6.18), so Thm. \ref{mainthm} has an analogue in the adic setting.  That is, we have the adic generic fiber $\tilde{G}^{n,\ad}_{\overline{\eta}}$ (resp., $\wedge\tilde{G}^{\ad}_{\overline{\eta}}$) of the formal scheme $\tilde{G}^n_{\OO_C}$ (resp., $\wedge\tilde{G}_{\OO_C}$), and the adic generic fiber $\mathcal{M}_{\wedge G_0,\overline{\eta}}^{\ad}$ of $\mathcal{M}_{\wedge G_0}$.  All of these are perfectoid spaces over $\C$, and $\mathcal{M}^{\ad}_{\overline{\eta}}$ is the fiber product of  $\tilde{G}^{n,\ad}_{\overline{\eta}}$ and $\mathcal{M}_{\wedge G_0,\overline{\eta}}^{\ad}$ over $\wedge\tilde{G}^{\ad}_{\overline{\eta}}$.

\subsection{CM points}
\subsection{CM points}
Let $L/K$ be an extension of degree $n$, with uniformizer $\pi_L$ and residue field $\FF_Q$.  Let $C$ be a valued field containing $L$.   A deformation $H$ of $H_0$ over $\OO_C$ has {\em CM by $\OO_L$} if there exists a $K$-linear isomorphism $L\to \End G\otimes K$.  Equivalently, $H$ has CM by $L$ if it is isogenous to a formal $\OO_L$-module (necessarily of height 1).

Fix such a CM deformation $H$ over $\OO_{\hat{L}^{\nr}}$.  Let $L_\infty$ be the field obtained by adjoining the $\pi^m$ torsion of $H$ to $\hat{L}^{\nr}$.  If $\phi$ is a level structure on $H$, we get a triple $(H,\rho,\phi)$ defining an $\hat{L}_\infty$-point of $\M_\infty$.   Points of $\mathcal{M}_{H_0,\infty,C}$ constructed in this matter will be called {\em CM points} (or points with CM by $L$).

Let $x$ be a point of $\M_\infty$ corresponding to the triple $(H,\iota,\phi)$.  Then $x$ induces embeddings $j_{x,1}=\from L\to M_n(K)$ and $j_{x,2}\from L\to D$, characterized by the commutativity of the diagrams
\[
\xymatrix{
K^n \ar[d]_{j_{x,1}(\alpha)} \ar[r]^{\phi} & V(H) \ar[d]^{\alpha} \\
K^n \ar[r]_{\phi} & V(H)
}
\]
and
\[
\xymatrix{
H_0 \ar[d]_{j_{x,2}(\alpha)} \ar[r]^{\iota} & H\otimes \FF_Q \ar[d]^{\alpha} \\
H_0 \ar[r]_{\iota} & H\otimes\FF_Q
}
\]
for $\alpha\in K$.

The group $\GL_n(K)\times D^\times$ acts transitively on the set of CM points;  the stabilizer of any particular CM point $x$ is the image of $L^\times$ under the diagonal map $\alpha\mapsto (j_{x,1}(\alpha),j_{x,2}(\alpha))$.  If the CM point $x$ is fixed, we will often suppress $j_{x,1}$ and $j_{x,2}$ from the notation and think of $L$ as a subfield of both $M_n(K)$ and $D$.

By Lubin-Tate theory, points of $\M_{H_0,\infty}$ with CM by $L$ are defined over the completion of the maximal abelian extension of $L$.  That is, these points are fixed by the commutator $[W_L,W_L]$.  Recall that the relative Weil group $W_{L/K}$ is the quotient of $W_K$ by the closure of the commutator $[W_L,W_L]$.  If $x$ has CM by $L$, and $w\in W_K$, then $x^w$ also has CM by $L$, and therefore there exists a pair $(g,b)\in\GL_n(K)\times D^\times$ for which $x^{(g,b)}=x^{w^{-1}}$.  Then $w\mapsto L^\times (g,b)$ is a well-defined map $j_x\from W_{L/K}\to L^\times\backslash(\GL_n(K)\times D^\times)$ (where the $L^\times$ is diagonally embedded).

Recall also that there is an exact sequence
\[
\xymatrix{
1\ar[r] & L^\times \ar[r]^{\rec_L} & W_{L/K} \ar[r] & \Gal(L/K)
}
\]
correponding to the canonical class in $H^2(\Gal(L/K),L^\times)$ (cf. \cite{TateNumberTheoreticBackground}).

\begin{lemma}  For all $\alpha\in L^\times$, we have $j_x(\rec_L\alpha)=L^\times(1,\alpha)$.
\end{lemma}

\begin{proof}  This is tantamount to the statement that $x^{\rec_L\alpha}=x^{(1,\alpha^{-1})}$, and will follow from classical Lubin-Tate theory.

By replacing $x$ with a translate, we may assume that $H/\OO_L$ admits endomorphisms by all of $\OO_K$, and that $\phi$ maps $\OO_K^n$ isomorphically onto $T(H)$.

In~\cite{LubinTate}, the main theorem shows that the maximal abelian extension $L^{\ab}/L$ is the compositum of $L^{\nr}$, the maximal unramified extension, with $L_\infty$, the field obtained by adjoining the roots of all iterates $[\varpi^n]_{H}$ to $L$.   Write $\alpha=\pi_L^mu$, with $u\in\OO_L^\times$.
In the notation of~\cite{LubinTate}, the Artin symbol $(\alpha,L^{\ab}/L)$ restricts to the $m$th power of the (arithmetic) $Q$th power Frobenius $\Fr_{Q}$ on $L^{\nr}$ and sends a root $\xi$ of $[\pi^n_L]_{H}$ to $[u^{-1}]_{H}(\xi)$.
But $\rec_L$ sends a uniformizer to a geometric Frobenius, so $\rec_L(\alpha)=(\alpha^{-1},L^{\ab}/L)$ as elements of $W_L^{\ab}$.
Thus for a unit $u\in \OO_L^\times$, $x^{\rec_L(u)}$ is represented by the triple $(H,\iota,\phi)^{\rec_L(u)}=(H,\iota,\phi\circ u)=x^{(u,1)}=x^{(1,u^{-1})}$ as claimed.
Finally, since $\rec_L(\varpi_L)$ acts as geometric Frobenius on $L^{\nr}$ and as the identity on $L_\infty$, we have  $x^{\rec_L(\pi_L)}=(H,\iota,\phi)^{\rec_L(\pi_L)}=(H,\iota\circ\Fr_{Q}^{-1},\phi)$.
Since $[\pi_L]_{G}$ reduces to $\Fr_{Q}$ modulo $\varpi$, we have $x^{\rec_L(\varpi)}=x^{(1,\varpi^{-1})}$.  This completes the proof of the claim.
\end{proof}

Let $\sigma$ be the image of $w$ in $\Gal(L/K)$, let $\tilde{j}_x(w)=L^\times(g,b)$, and let $\alpha\in L^\times$ be arbitrary.  Repeatedly using the fact that the actions of $\GL_n(K)\times D^\times$ and $W_K$ on $\mathcal{M}_{H_0,\infty}(\OO_C)$ commute, we have by the previous lemma
\begin{eqnarray*}
x^{(1,b^{-1}\alpha b)}
&=&x^{(g,b)^{-1}(1,\alpha)(g,b)}\\
&=& x^{w(1,\alpha)(g,b)} \\
&=& x^{(1,\alpha)w(g,b)}\\
&=& x^{\rec_L(\alpha)w(g,b)}\\
&=& x^{(g,b)\rec_L(\alpha)w}\\
&=& x^{w^{-1}\rec_L(\alpha)w }\\
&=& x^{\rec_L(\alpha^\sigma)}\\
&=& x^{(1,\alpha^{\sigma})},
\end{eqnarray*}
so that $b^{-1}\alpha b=\alpha^{\sigma}$.  A similar calculation shows that $g^{-1}\alpha g=\alpha^{\sigma}$.   Let $\mathcal{N}$ and $\mathcal{N}_D$ be the normalizers of $L^\times$ in $\GL_n(K)$ and $D^\times$, respectively.  Then both $\mathcal{N}$ and $\mathcal{N}_D$ are extensions of $\Gal(L/K)$ by $L^\times$.  Let $\tilde{\mathcal{N}}\subset \GL_n(K)\times D^\times$ be the pushout in the diagram
\[
\xymatrix{
\tilde{\mathcal{N}} \ar[r] \ar[d]& \mathcal{N}_D \ar[d] \\
\mathcal{N} \ar[r]& \Gal(L/K).
}
\]
The above calculation shows that $(g,b)$ lies in $\tilde{\mathcal{N}}$.   Thus $j$ takes values in the {\em group} $L^\times\backslash\tilde{\mathcal{N}}$, and one sees from the calculation
\[
x^{j(w)j(w')}=x^{w^{-1}j(w')}=x^{j(w')w^{-1}}=x^{(w')^{-1}w^{-1}}=x^{(ww')^{-1}}=x^{j(ww')}
\]
that $j\from W_K\to L^\times\backslash\tilde{\mathcal{N}}$ is a group homomorphism.  The restriction of $j$ to $W_L$ factors through $W_L^{\ab}$, so that $j$ factors through $W_K/[W_L,W_L]^{\text{cl}}=W_{L/K}$.  From the diagram
\[
\xymatrix{
1 \ar[r] & W_L^{\ab} \ar[d]_{\rec_L^{-1}} \ar[r] & W_{L/K}  \ar[r] \ar[d]^{j} & \Gal(L/K) \ar[d]^{=} \ar[r] & 1 \\
1 \ar[r] &  L^\times \ar[r]^{(1,\alpha)} &  \tilde{\mathcal{N}}/L^{\times} \ar[r] & \Gal(L/K) \ar[r] & 1
}
\]
we see that $j\from W_{L/K}\to L^\times\backslash\tilde{\mathcal{N}}$ is an isomorphism.  This isomorphism is characterized by the property that
\begin{equation}
\label{jproperty}
x^{j(w)}=x^{w^{-1}}.
\end{equation}

Observe that $\mathcal{N}$ is a {\em split} extension of $\Gal(L/K)$ by $L^\times$.  (This is clear as soon as we identify $G$ with $\Aut_K L$, and $\mathcal{N}$ as the semidirect product of $L^\times$ (acting by left multiplication on $L$) and $\Gal(L/K)$).  Choose a splitting $\Gal(L/K)\to \GL_n(K)$, and identify $\Gal(L/K)$ with its image.  There is an isomorphism $\tilde{\mathcal{N}}/L^\times\to\mathcal{N}_D$ given by sending $(\alpha \sigma,b)$ to $\alpha^{-1}b$ for every triple $\alpha\in L^\times$, $\sigma\in\Gal(L/K)$, $b\in D^\times$ such that conjugation by $b$ acts on $L^\times$ via $\sigma$.  Write $j_D\from W_{L/K}\to\mathcal{N}_D$ for the isomorphism which fits into the commutative diagram
\[
\xymatrix{
& L^\times\backslash\tilde{\mathcal{N}} \ar[dd]^{\sim} \\
W_{L/K} \ar[ur]^{\tilde{j}} \ar[dr]_{j_D} & \\
& \mathcal{N}_D.
}
\]

\begin{prop}
\label{stabx}
Let $\mathcal{S}$ be the stabilizer of $x$ in $\GL_n(K)\times D^\times\times W_K$.  Then $\mathcal{S}$ is generated by the following collections of elements:
\begin{enumerate}
\item $\set{(\alpha,\alpha,1)\biggm\vert \alpha\in L^\times}$,
\item $(j(w),w)$, where $w\in W_K$ (recall that $j(w)$ lies in $\tilde{\mathcal{N}}\subset \GL_n(K)\times D^\times$).
\end{enumerate}
\end{prop}

\begin{proof} We have already seen that these elements fix $x$.  Suppose $(g,b,w)$ fixes $x$, so that $x^{(g,b)}=x^{w^{-1}}=x^{j(w)}$.  Then $(g,b)j(w)^{-1}\in GL_n(K)\times D^\times$ fixes $x$, and so must lie in the diagonally embedded $L^\times$.  
\end{proof}

\subsection{A family of special affinoids in $\mathcal{M}^{\ad}_{\overline{\eta}}$}
Let $L/K$ be the unramified extension of degree $n$, so that $L=\FF_{q^n}\laurentseries{\pi}$.  Let $x_{\CM}=(x_1,\dots,x_n)\in \mathcal{M}_{G_0}(\OO_{\C})$ be a point with CM by $\OO_L$.  Finally, let $m\geq 1$ be odd.  We will construct an open affinoid $\mathcal{Z}=\mathcal{Z}_{x_{\CM},m}\subset \mathcal{M}^{\ad}_{\overline{\eta}}$, depending on $x_{\CM}$ and $m$.  This affinoid will be the intersection of $\mathcal{M}^{\ad}_{\overline{\eta}}$ with a certain open affinoid $\mathcal{Y}\subset \tilde{G}^{n,\ad}_{\overline{\eta}}$.  Here $\mathcal{Y}$ is isomorphic to the ``perfectoid closed ball" $\Spa(R,R^+)$, where $R^+=\OO_{\C}\tatealgebra{Y_1^{1/q^\infty},\dots,Y_n^{1/q^\infty}}$ and $R=R^+[1/\pi]$.  We will compute an approximation to the restriction of the determinant map $\delta\from \tilde{G}^{n,\ad}_{\overline{\eta}}\to \wedge\tilde{G}^{\ad}_{\OO_{\C}}$ to $\mathcal{Y}$.  This calculation will enable us to compute the reduction of $\mathcal{Z}$.  That is, if $\mathcal{Z}=\Spa(S,S^+)$, we give an explicit description of $S^+\otimes_{\OO_{\C}}\overline{\FF}_q$.

It will be simpler for us to assume that there exist $x\in\tilde{G}(\OO_{\C})$ and elements $\alpha_1,\dots,\alpha_n\in \FF_q$ such that $x_i=\alpha_i x$.  This is equivalent to the assumption that the embedding $\OO_L\injects M_n(\OO_K)$ carries $\FF_{q^n}$ into $M_n(\FF_q)$.  Since $\OO_{\C}$ is perfect, $\tilde{G}(\OO_{\C})$ is in bijection with the set of topologically nilpotent elements in $\OO_{\C}$.  We will intentionally confuse $x_i$ and $x$ with the elements of $\OO_{\C}$ to which they correspond.  Similarly, the exterior product $x_1\wedge\cdots\wedge x_n\in\wedge\tilde{G}(\OO_{\C})$ corresponds to a topologically nilpotent element $t\in\OO_{\C}$.  Note that $\abs{\pi}=\abs{x}^{q^n-1}=\abs{t}^{q-1}$.

We now describe the affinoid $\mathcal{Y}$.  $\tilde{G}^{n,\ad}_{\OO_{\C}}$ is the adic generic fiber of $\Spa(B,B)$, where $B=\OO_{\C}\powerseries{X_1^{1/q^\infty},\dots,X_n^{1/q^\infty}}$.  Define elements $Y_1,\dots,Y_n\in B[1/\pi]$ by the system of $n$ equations
\begin{equation}
\label{XfromY}
X_i=x_i+(x_iY_1)^{q^{\frac{(m-1)n}{2}+1}}+(x_iY_2)^{q^{\frac{(m-1)n}{2}+2}}+\dots+
(x_iY_{n-1})^{q^{\frac{(m-1)n}{2}+n-1}}+(x_iY_n)^{q^{mn}}.
\end{equation}
(The system is nondegenerate because the $x_i$ are linearly independent over $\FF_q$.)  We define $\mathcal{Y}$ by the conditions $\abs{Y_i}\leq 1$, $i=1,\dots,n$.  This is a rational subset of $\tilde{G}^{n,\ad}_{\OO_{\C}}$.  We have $\mathcal{Y}=\Spa(R[1/\pi],R)$, where
\[ R=\OO_{\C}\tatealgebra{Y_1^{1/q^\infty},\dots,Y_n^{1/q^\infty}}\]
and $R=R^+[1/\pi]$.

The ring $R$ is perfect, so that $\tilde{G}(R)$ may be identified with the set of topologically nilpotent elements $\Nil(R)$ of $R$.   The elements $X_1,\dots,X_n\in R$ are topologically nilpotent and therefore determine an $n$-tuple in $\tilde{G}^n(R)$.  We may take their exterior product to get an element of $\wedge \tilde{G}(R)$, which we identify with a topologically nilpotent element $\delta(X_1,\dots,X_n)\in R$.

We intend to give an approximation of $\delta(X_1,\dots,X_n)$ in terms of the variables $Y_1,\dots,Y_n$.  For this, we recall the construction of the determinant of a formal vector space.  We have the isocrystal $M^*(\tilde{G}_0)$:  this is a $\hat{K}^{\nr}$-vector space of dimension $n$ with a Frobenius-semilinear automorphism $F$.  It admits a basis $e,Fe,\dots,F^{n-1}e$, with $F^ne=\pi^{-1}e$.  We have an isomorphism \[v\from \tilde{G}(R)\tilde{\to} \left(M^*(\tilde{G}_0)\hat{\otimes}_{\overline{\FF}_q} \Nil(R)\right)^{F\otimes\tau},\]
which carries $z\in \Nil(R)$ onto $v(z)=\sum_{i\in\Z} F^ie\otimes z^{q^i}$.
The top exterior power $\wedge M^*(\tilde{G}_0)$ has basis $f=e\wedge\cdots\wedge F^{n-1}e$, on which $F$ acts as $(-1)^{n-1}\pi^{-1}$.  We have the elements $X_1,\dots,X_n\in \Nil(R)^n$, and we can form the exterior product
\[v(X_1)\wedge\cdots\wedge v(X_n)=\sum_{i\in\Z} F^if \otimes \delta(X_1,\dots,X_n)^{q^i}.\]   So the element $\delta(X_1,\dots,X_n)\in\Nil(R)$ we seek is the coefficient of $f$ in
$v(X_1)\wedge\cdots\wedge v(X_n)$.

Let
\[ W = e\otimes x + \sum_{r=1}^{n-1}F^{-\frac{(m-1)n}{2}-r}e\otimes xY_r+F^{-mn}e\otimes xY_n, \]
and for $j=0,\dots,n-1$ let
\[ U_j = \sum_{i\equiv j\;(n)} \left(F^i\otimes \tau^i\right)(W), \]
so that for $j=1,\dots,n$ we have
\[
\sum_{i=1}^n\alpha_i^{q^j}U_j = v(X_j). \]
Let $\varepsilon=\det(\alpha_i^{q^j})$, so that $v(X_1)\wedge\cdots v(X_n)=\varepsilon U_0\wedge\cdots\wedge U_{n-1}$.


\begin{lemma} \label{nonconstant}
Up to a sign, the sum of the nonconstant terms in the coefficient of $f$ in $U_0\wedge\cdots \wedge U_{n-1}$ are congruent modulo $x^{q^{m+n}}R$ to the sum of the nonconstant terms in the coefficient of $f$ in $\bigwedge_{i=m}^{m+n-1} \left(F^i\otimes\tau^i\right)(W)$.
\end{lemma}

\begin{proof} In the expansion of $U_0\wedge\cdots \wedge U_{n-1}$ we encounter terms of the form $\bigwedge_{i=0}^{m-1} \left(F^{b_i}\otimes\tau^{b_i}\right)(W)$, where $(b_0,\dots,b_{n-1})$ is an $n$-tuple of integers with $b_i\equiv i\pmod{n}$.  After substituting the expression for $W$, we find that the coefficient of $f$ in $\bigwedge_{i=0}^{m-1} \left(F^{b_i}\otimes\tau^{b_i}\right)(W)$ is sum of terms of the form $T=\bigwedge_{i=0}^{m-1}\left(F^{b_i}\otimes\tau^{b_i}\right)(F^{-s_i}e\otimes xz_i)$.  Here $s_i$ assumes one of the values $0,(m-1)n/2+1,\dots,(m-1)n/2+n-1,mn$, with $0$ occurring only if and only if $z_i$ is a constant.  These must satisfy $\sum_i (b_i-s_i)=n(n-1)/2$.  Furthermore the $b_i-s_i$ must be distinct modulo $n$, so that if any $s_i$ assumes the value $(m-1)n/2+r$, then there must be a $j$ with $s_j=(m-1)n/2+r'$.

If $T$ is nonzero modulo $x^{q^{m+n}}R$, then we have $b_i\leq m+n-1$ for $i=1,\dots,n-1$.  Since the $b_i$ are distinct modulo $n$, we have $\sum_i b_i\leq \sum_{i=m}^{m+n-1}i=mn+n(n-1)/2$.   Thus $\sum_i s_i\leq mn$.  If in addition $T$ is nonconstant, then at least one of the $s_i$ is nonzero.  Thus either one of the $s_i$ equals $mn$, or else we have $s_i=(m-1)n/2+r$ and $s_j=(m-1)n/2+r'$ for some $i$ and $j$.  Since $\sum_i s_i$ is divisible by $n$, this shows at once that $\sum_i s_i=mn$.  Therefore $\sum_i b_i=mn+n(n-1)/2$, which implies that the $b_i$ are the integers $m,m+1,\dots,m+n-1$.
\end{proof}

Let $\psi$ be the $K\hat{\otimes}_{\overline{\FF}_q}R$-linear endomorphism of $M^*(\tilde{G}_0)\hat{\otimes}_{\overline{\FF}_q}R$ whose matrix is
\[
\label{psimatrix}
\text{Id}_{n\times n} +
\begin{pmatrix}
\pi^m\otimes Y_n^{q^m} & \pi^{(m-1)/2}\otimes Y_1^{q^m} & \dots & \pi^{(m-1)/2}\otimes Y_{n-1}^{q^m}  \\
\pi^{(m+1)/2}\otimes Y_{n-1}^{q^{m+1}} & \pi^m\otimes Y_n^{q^{m+1}} & \dots & \pi^{(m-1)/2}\otimes Y_{n-1}^{q^{m+1}} \\
\vdots & \vdots & \ddots & \vdots \\
\pi^{(m+1)/2}\otimes Y_1^{q^{m+n-1}} & \pi^{(m+1)/2}\otimes Y_2^{q^{m+n-1}} & \cdots & \pi^m\otimes Y_n^{q^{m+n-1}}
\end{pmatrix}
\]
with respect to the basis $e\otimes 1,\dots,F^{n-1}e\otimes 1$.   Then for $j=m,\dots,m+n-1$ we have $\psi(F^je\otimes x^{q^j})=\left(F^j\otimes\tau^j\right)(W)$.  The determinant of the matrix in Eq.~\eqref{psimatrix} takes the form $1+\pi^m\otimes N(Y_1,\dots,Y_n)^{q^m}$ modulo $\pi^{m+1}\OO_{\hat{K}^{\nr}}\otimes R$, for a polynomial $N(Y_1,\dots,Y_n)\in \FF_q\powerseries{Y_1,\dots,Y_n}$.  We have
\[\bigwedge_{j=m}^{m+n-1}\psi(F^je\otimes x^{q^j})=\left(\det\psi\right)\left(\pi^{-m} f\otimes x^{(1+q+\dots+q^{n-1})m}\right).\]
Thus the coefficient of $f$ in $\bigwedge_{j=m}^{m+n-1}\left(F^j\otimes\tau^j\right)(W)$ is $N(Y_1,\dots,Y_n)^{q^m}x^{(1+q+\dots+q^{n-1})m}$.
Combining this with Lemma~\ref{nonconstant} gives

\begin{prop}\label{congruenceN} We have the congruence
\[ \delta(X_1,\dots,X_n)\equiv t+N(Y_1,\dots,Y_n)^{q^m}t^{q^m} \pmod{t^{q^m+\varepsilon}R}\]
for some $\varepsilon>0$.
\end{prop}

Recall from Eq.~\ref{Acirc} that $A_{\OO_C}^{\circ}=\OO_C\powerseries{X_1^{1/q^\infty},\dots,X_n^{1/q^\infty}}/(\delta(X_1,\dots,X_n)^{1/q^r}-t^{1/q^r})_{r\geq 0}$.  Let $\mathcal{Z}=\mathcal{Y}\cap\mathcal{M}^{\ad}_{\overline{\eta}}$ and $\mathcal{Z}^{\circ}=\mathcal{Y}\cap\mathcal{M}^{\circ,\ad}_{\overline{\eta}}$.
Let $S'=A_{\OO_C}^{\circ}\tatealgebra{Y_1,\dots,Y_n}$.   Then $\mathcal{Z}^{\circ}=\Spa(S,S^+)$, where $S=S'[1/\pi]$, and $S^+$ is the integral closure of $S'$ in $S$.  Let $d=(\delta(X_1,\dots,X_n)-t)/t^{q^m}$.  We have
\[ S'=\OO_{\C}\tatealgebra{Y_1^{1/q^\infty},\dots,Y_n^{1/q^\infty}}/\left(d^{1/q^r}\right)_{r\geq 0}
\]
Since this ring is perfect, it is integrally closed in $S$, and so $S^+=S'$.  From this description of $S^+$ and the congruence in Prop. \ref{congruenceN} we get a description of $\Spec (S^+\otimes_{\OO_C} \overline{\FF}_q)$, which is the reduction of the affinoid $\mathcal{Z}^{\circ}$.  

\begin{prop}
\label{Zreduction}
The reduction of $\mathcal{Z}^{\circ}$ is the perfection of the nonsingular variety $N(Y_1,\dots,Y_n)=0$.
\end{prop}

If $m$ is even, the definition of the affinoid $\mathcal{Y}$ is as follows.  Identify $M_n(K)$ with $\End_K(L)$, so that $M_n(K)$ is spanned over $L$ by $1,\sigma,\dots,\sigma^{n-1}$, where $\sigma\in \Gal(L/K)$ is the Frobenius automorphism.  For $j=1,\dots,n-1$, let $x^{\sigma^j}_i$ be the $i$th coordinate of $x^\sigma$.  Define variables $Y_1,\dots,Y_n\in B[1/\pi]$ by 
\[ X_i = x_i + (x_i^\sigma Y_1)^{q^{mn/2}}+\dots+(x_i^{\sigma^{n-1}}Y_{n-1})^{q^{mn/2}}+(x_iY_n)^{q^{mn}},\]
and then define the affinoid $\mathcal{Y}\subset\tilde{G}^{n,\ad}_{\overline{\eta}}$ by the conditions $\abs{Y_i}\leq 1$.  Then Props. \ref{congruenceN} and \ref{Zreduction} go through as in the case of $m$ odd, but possibly with a different polynomial $N$.  In all cases, $N(Y_1,\dots,Y_n)$ is the coefficient of $\pi^m$ in the determinant of the matrix
\[ \text{Id}_{n\times n} + 
\begin{pmatrix}
1+\pi^mY_n & \pi^{\floor{m/2}}Y_1 & \cdots & \pi^{\floor{m/2}}Y_n \\
\pi^{\ceil{m/2}}Y_{n-1}^q & 1+\pi^mY_n^q & \cdots & \pi^{\floor{m/2}} \\
\vdots & \vdots & \ddots & \vdots \\
\pi^{\ceil{m/2}}Y_1^{q^{n-1}} & \pi^{\ceil{m/2}}Y_2^{q^{n-1}} & \cdots & 1+\pi^mY_n^{q^{n-1}}
\end{pmatrix}
\]
Observe that $N(Y_1,\dots,Y_n)$ only depends on whether $m=1$ or $m\geq 2$.  (Furthermore, when $n=2$, we have $N(Y_1,Y_2)=Y_1+Y_1^q-Y_2^{q+1}$ no matter the value of $m$.)

\subsection{The stabilizer of $\mathcal{Y}$}

As in the previous subsection, $L/K$ is the unramified extension of degree $n$, and $x_{\CM}\in\mathcal{M}(\OO_C)$ is a CM point for $L$.  Then we have the affinoids $\mathcal{Y}\subset \tilde{G}^{n,\ad}_{\overline{\eta}}$ and $\mathcal{Z}^{\circ}\subset\mathcal{M}^{\circ,\ad}_{\overline{\eta}}$.  We record the relationship between these affinoids and the actions of the triple product group $\GL_n(K)\times D^\times\times W_L$.  The details aren't difficult to work out because we have explicit formulas for the actions of $\GL_n(K)$ and $D^\times$ on $\tilde{G}^n$ and for the action of $W_K$ on the CM point $x_{\CM}$.  First we describe the stabilizer of $\mathcal{Y}$ inside the groups $\GL_n(K)$ and $D^\times$.

The point $x_{\CM}$ induces embeddings of $\OO_L$ into $M_n(\OO_K)$ and $\OO_D$.    Let $C$ (respectively, $C_D$) be the complement of $\OO_L$ inside of $M_n(\OO_K)$ (respectively, $\OO_D$) with respect to the natural trace pairing.  If we identify $M_n(\OO_K)$ with the ring of $\OO_K$-linear endomorphisms of $\OO_L$, then $C$ is the span of $\sigma,\dots,\sigma^{n-1}$, where $\sigma$ is the Frobenius in $\Gal(L/K)=\Gal(\FF_{q^n}/\FF_q)$.  Meanwhile, $C_D$ is the span of $\varpi,\dots,\varpi^{n-1}$.

Define subgroups $U\subset \GL_n(\OO_K)$ and $U_D\subset \OO_D^\times$ by
\begin{eqnarray*}
U &=& 1+\gp_L^m+\pi_L^{\ceil{m/2}}C \\
U_D &=& 1+\gP_L^m + \pi_L^{\floor{m/2}}C_D.
\end{eqnarray*}
Then $U$ and $U_D$ map $\mathcal{Y}$ onto itself.  



Recall that $\mathcal{S}$ is the stabilizer of $x$ in $\GL_n(K)\times D^\times\times W_K$.  From Prop. \ref{stabx} it is easy to see that $\mathcal{S}$ normalizes $U\times U_D\times\set{1}$.  Let $\mathcal{J}\subset\GL_n(K)\times D^\times\times W_K$ be the subgroup generated by $U\times U_D\times\set{1}$ and $\mathcal{S}$.  

\begin{prop} Every element of $\mathcal{J}$ maps $\mathcal{Y}$ onto itself.  Conversely, an element of $\GL_n(K)\times D^\times\times W_{K}$ which maps a point in $\mathcal{Y}$ onto another point in $\mathcal{Y}$ must lie in $\mathcal{J}$.
\end{prop}

\subsection{Description of the group actions on $\overline{\mathcal{Y}}$.}

We can also give a description of the actions of $U$ and $U_D$ on the reduction $\overline{\mathcal{Y}}=\Spec \overline{\FF}_q[Y_1^{1/q^\infty},\dots,Y_n^{1/q^\infty}]$.  
Define a subring $\mathcal{L}$ of $M_n(\OO_K)\times \OO_D$ by
\[ \mathcal{L} = \Delta(\OO_L)+\left(\gp_L^m\times\gp_L^m\right)+\left(\gp_L^{\ceil{m/2}}C\times\gp_L^{\ceil{(m-1)/2}}C_D\right). \]
Then $\mathcal{L}$ is an order in $M_n(L)\times D$ which contains $\Delta(\OO_L)$.

Let $\gP\subset\mathcal{L}$ be the double-sided ideal
\[ \gP=\Delta(\gp_L)+\left(\gp_L^{m+1}\times\gp_L^{m+1}\right)+\left(\gp_L^{\ceil{(m+1)/2}}C\times \gp_L^{\ceil{m/2}}\right).\]

\begin{lemma}
\label{Sring}
The quotient $S=\mathcal{L}/\gP$ is an $\FF_{q^n}$-algebra of dimension $n+1$.  It admits a basis $1,e_1,\dots,e_n$.  Multiplication in $S$ is determined by the following rules.  We have $e_i \alpha = \alpha^{q^i} e_i$, all $\alpha\in\FF_{q^n}$.  If $m=1$, then the product $e_ie_j$ is $e_{i+j}$ if $i+j\leq n$, and is $0$ otherwise.   If $m\geq 1$, then $e_ie_j$ is $e_n$ if $i+j=n$, and is $0$ otherwise.
\end{lemma}

\begin{proof}  This is a simple calculation.  An interesting feature is that the roles of $M_n(K)$ and $D$ alternate based on the parity of $m$.  If $m$ is even, then $e_i$ is the image of $(\pi^{m/2}\sigma^i,0)$.  If $m$ is odd, then $e_i$ is the image of $(0,\pi^{(m-1)/2}\varpi^i)$.
\end{proof}

(The family of orders which we have called $\mathcal{L}$ is best understood from the point of view of harmonic analysis on $M_n(K)\times D$.  See \cite{WeinsteinFourier} for a thorough discussion of these orders in the case $n=2$.)  

Let $\mathbf{U}$ be the affine group variety over $\overline{\FF}_q$ whose points over an $\overline{\FF}_q$-algebra $R$ are formal expressions $1+\alpha_1e_1+\dots+\alpha_ne_n$, with $\alpha_i\in R$.  The group operation is determined by the same rules as in Lemma~\ref{Sring}.  Thus there is a natural surjection $U\times U_D\to \mathbf{U}(\FF_{q^n})$.  Of course we have $\mathbf{U}=\Spec \overline{\FF}_q[Y_1,\dots,Y_n]$, for the obvious choice of coordinates $Y_1,\dots,Y_n$.  Let us identify $\overline{\mathcal{Y}}$ with the perfect closure of $\mathbf{U}$:  $\overline{\mathcal{Y}}=\mathbf{U}^{\text{perf}}=\overline{\FF}_q[Y_1^{1/q^\infty},\dots,Y_n^{1/q^\infty}]$.  

\begin{prop} 
\label{Jaction}
The (right) action of $\mathcal{J}$ on $\overline{\mathcal{Y}}$ is determined by the following properties.  The subgroup $U\times U_D\times\set{1}$ of $\mathcal{J}$ acts on $\overline{\mathcal{Y}}=\mathbf{U}^{\text{perf}}$ through its quotient $\mathbf{U}(\F_{q^n})$, which acts on $\mathbf{U}$ by right multiplication.
The action of the diagonally embedded $K^\times$ in $\GL_n(K)\times D^\times$ is trivial.  The action of elements of the form $(\alpha,\alpha,1)$ with $\alpha\in\OO_L^\times$ is through 
\[ (Y_1,\dots,Y_n)\mapsto (\overline{\alpha}^{q-1}Y_1,\dots,\overline{\alpha}^{q^{n-1}-1}Y_{n-1},Y_n), \]
where $\overline{\alpha}$ is the image of $\alpha$ in $\FF_{q^n}$.  Finally, if $\Phi\in W_K$ is a Frobenius element, then $(j(\Phi),\Phi)$ acts on $\overline{\mathcal{Y}}$ as the geometric Frobenius map $Y_i\mapsto Y_i^q$ (which is an automorphism).
\end{prop}

Then of course the stabilizer of $\mathcal{Z}^{\circ}$ in $\GL_n(K)\times D^\times\times W_K$ is $\mathcal{J}^{\circ}=\mathcal{J}\cap(\GL_n(K)\times D^\times\times W_K)^{\circ}$.  Prop. \ref{Jaction} gives a description of the action of $\mathcal{J}^{\circ}$ on the reduction $\overline{\mathcal{Z}}^\circ$, which by Prop. \ref{Zreduction} is the perfection of the variety $N(Y_1,\dots,Y_n)=0$.

\bibliographystyle{amsalpha}
\bibliography{RZspaces}

\end{document}